\documentclass{amsart}

\usepackage[utf8]{inputenc}
\usepackage[T1]{fontenc}
\usepackage{lmodern}
\usepackage{amsfonts,amssymb}
\usepackage[english]{babel}
\usepackage{amsmath}

\usepackage{amsthm}

\usepackage{epsfig}
\usepackage{mathrsfs}
\usepackage{xcolor}
\usepackage{graphicx}
\usepackage{dsfont}
\usepackage{enumerate}
\usepackage{cases}
\usepackage{bbm}
\usepackage{hyperref}
\usepackage{stmaryrd}

\def\RR{\mathbb{R}}

\def\div{\mathrm {div}\,}

\def\L{\mathcal L}
\def\Rx{\mathcal{R}_x }

\def\Rxf{\mathcal{R}_{x_f} }

\def\D{\mbox{\tiny D}}
\def\SR{\mbox{\tiny SR}}
\def\M{\mbox{\tiny M}}
\def\B{\mathcal{B}}

\def\BD{\mathcal{B_{\D}}}
\def\BSR{\mathcal{B_{\SR}}}
\def\BDs{\mathcal{B_{\D}^*}}
\def\BSRs{\mathcal{B_{\SR}^*}}

\def\pv{\text{P.V.}} 

\def\fin{f_{in}} 
\def\bO{\bar{\Omega}}
\def\dO{\pa\Omega}

\def\Dels{\big(-\Delta\big)^s}

\def\Delsreg{\big(-\Delta\big)_\Omega^s}
\def\DelsSR{(-\Delta)_{\text{\tiny{SR}}}^s}

\def\d{\, {\rm{d}} }

\def\pa{\partial}
\def\na{\nabla}
\def\eps{\varepsilon}

\def\HSRs{\mathcal{H}_{\textrm{\tiny{SR}}}^s }
\def\HSR2s{\mathcal{H}_{\textrm{\tiny{SR}}}^{2s} }

\def\Lsr*{\mathscr{L}_{\textrm{\tiny{SR}}}^*}

\def\Hdiff0{\mathcal{H}_{\textrm{\tiny{diff}},0}^s }

\def\Deps{\mathcal{D}^{2s-1}_\eps}
\def\Dlim{\mathcal{D}^{2s-1}}

\def\tfe{\tau_f^\eps }

\def\keps{\kappa_{\D}^\eps}

\theoremstyle{plain}
\newtheorem{thm}{Theorem}[section]
\newtheorem*{thm*}{Theorem}
\newtheorem{lemma}[thm]{Lemma}
\newtheorem*{lemma*}{Lemma}
\newtheorem{prop}[thm]{Proposition}
\newtheorem*{prop*}{Proposition}
\newtheorem{defi}{Definition}[section]
\newtheorem*{defi*}{Definition}

\newtheorem{remark}[thm]{Remark}
\newtheorem{cor}[thm]{Corollary}

\makeatletter
\newtheorem*{rep@theorem}{\rep@title}
\newcommand{\newreptheorem}[2]{%
\newenvironment{rep#1}[1]{%
 \def\rep@title{#2 \ref{##1}}%
 \begin{rep@theorem}}%
 {\end{rep@theorem}}}
\makeatother

\newreptheorem{prop}{Proposition}{\bf}{\rm}
\newreptheorem{thm}{Theorem}{\bf}{\rm}
\newreptheorem{cor}{Corollary}{\bf}{\rm}
\newreptheorem{lemma}{Lemma}{\bf}{\rm}

\begin{document}

\title[Fractional diffusion limit of a linear Boltzmann model]{Fractional diffusion limit of a linear Boltzmann model with reflective boundaries in a half-space}
\date{}
\author{Ludovic Cesbron}
\address{D\'{e}partement de Math\'{e}matiques, CY Cergy Paris Universit\'{e}, France.}
\email{ludovic.cesbron@u-cergy.fr}

\begin{abstract}
We investigate the fractional diffusion limit of a Linear Boltzmann equation with heavy-tailed velocity equilibrium in a half-space with Maxwell boundary conditions. We derive a new confined version of the fractional Laplacian and show uniqueness of weak solution to the associated non-local diffusion equation.
\end{abstract}

\maketitle 
\tableofcontents

\tableofcontents

\section{Introduction}

This paper is concerned with the derivation of fractional diffusion equations in bounded domains from kinetic models. Since the pioneer works of Mellet-Mischler-Mouhot \cite{Mellet11} and Jara-Komorowski-Olla \cite{Jara2009} we know that fractional diffusion equations can be obtained as the long-time small-mean-free-path asymptotic regime of some kinetic models. In the present paper, we study the asymptotic behaviour of such kinetic models in bounded domains in order to derive confined versions of non-local diffusion equations, where the confinement and the interaction with the boundary is entirely deduced from the kinetic boundary condition. This line of research was initiated in \cite{Cesbron18} and extended e.g. in \cite{Aceves17,CesbronMelletPuel,CesbronMelletPuel2} where several types of kinetic boundary conditions have been considered. \\
We consider a linear Boltzmann-type equation in a bounded domain $\Omega$ subset of $\RR^d$ : 
\begin{equation} \label{eq:LinBoltz} 
	\left\{ \begin{aligned} 
		& \pa_t f +  v\cdot \na_x f  = L(f) &\text{ in } (0,+\infty) \times\Omega\times\RR^d \\
		& f (0,x,v) = f_{in} (x,v) &\text{ in } \Omega\times\RR^d \\
	\end{aligned} \right. 
\end{equation}
where the collision operator $L$ is a scattering operator
\begin{equation} \label{eq:defLB} 
	\begin{aligned}
		L(f) (v)&  = \nu_0 ( \rho F - f) 
	\end{aligned}
\end{equation}
with $\nu_0>0$ and $\rho = \int_{\RR^d} f \d v$. Throughout this paper, the thermodynamical equilibrium $F$ will be a normalised heavy-tail distribution function satisfying
\begin{equation}\label{eq:F0}
	\left\{ \begin{aligned}
		& F \in L^{\infty}, \quad \int F(v) \d v = 1, \quad F(v) = F(|v|)\geq 0 \\
		& \Big| F(v) - \frac{\gamma}{|v|^{d+2s}} \Big| \leq \frac{C}{|v|^{d+4s}} \qquad \mbox{ for all } |v|\geq 1.
	\end{aligned} \right. 
\end{equation}
This kinetic equation models the evolution of a particle distribution function $f(t,x,v)\geq 0$ depending on time $t>0$, position $x\in\Omega$ and velocity $v\in\RR^d$. The left-hand-side of the equation \eqref{eq:LinBoltz} models the free transport of particles -- notice that we do not consider any electric of magnetic field in this model -- whereas the scattering operator $L$ on the right-hand-side models the diffusive and mass preserving interaction between the particles and the background. \\
On the boundary of $\Omega$, we prescribe the behaviour of the particles in order for \eqref{eq:LinBoltz} to be well-posed. These boundary condition take the form of a balance between the out-going and in-going particles, hence we introduce the sets 
\begin{align*} 
	\Gamma_\pm := \Big\{ (x,v)\in\dO\times\RR^d :\, \pm n(x) \cdot v >0 \Big\} 
\end{align*}
where $n(x)$ is the outward unit normal vector to $\dO$ at $x\in\dO$. Writing $\gamma_+ f$ the restriction of the trace of $f$ to the out-going trace set $\Gamma_+$, and $\gamma_- f$ the restriction to the in-going set $\Gamma_-$, the boundary condition takes the form 
\begin{align} \label{def:generalBC}
	&\gamma_- f (t,x,v) = \B_\alpha [ \gamma_+ f] (t,x, v) &\text{ on }  (0,+\infty) \times \Gamma_-
\end{align}
We write $\B_\alpha$ the Maxwell boundary operator defined as a combination of specular and diffuse reflections: for some $\alpha\in [0,1]$:
\begin{equation} \label{def:MaxwellBC} 
	\B_\alpha [\gamma_+ f ](t,x,v) = \alpha \BD [\gamma_+ f] + (1-\alpha) \BSR [\gamma_+ f_\eps ] 
\end{equation}
with the specular reflections boundary operator given by 
\begin{equation}\label{def:SpecularBC}
	\BSR [\gamma_+ f ] (t,x,v) = \gamma_+ f(t,x, \Rx v) = \gamma_+ f\big( t,x, v - 2(v\cdot n(x)) n(x) \big)
\end{equation}
and the diffuse reflections boundary operator given by 
\begin{equation} \label{def:DiffuseBC}
	\BD[\gamma_+ f ](t,x,v) = c_0 F(v) \int_{w\cdot n(x)>0} \gamma_+ f (t,x,w) |w\cdot n(x) | \d w .
\end{equation}
Note that the constant $c_0$ in the diffusive boundary condition is a normalising constant that ensures that the equilibrium $F$ satisfies the boundary condition, i.e. 
\begin{align*}
	c_0 = \bigg(\int_{w\cdot n(x)<0} F(w) |w\cdot n(x) | \d w \bigg)^{-1}
\end{align*}
so we want the first moment of $F$ to be finite in order for this boundary condition to make sense. By assumptions \eqref{eq:F0} we know that exactly $2s$-moments of $F$ are finite so we shall assume that $s>1/2$ when considering the diffuse reflections condition. \\
These boundary conditions were introduced by Maxwell in \cite{Maxwell} in order to model the interaction between a particle and a boundary surface. The specular reflections boundary condition models a billiard-like reflection, it can be seen as a first approximation of the reflection process where the boundary is assumed to be a perfectly smooth surface without any minute asperities. The diffusive boundary condition is then a correction of this smoothness assumption, it can be derived by considering the boundary to be a stratum of particles whose velocities are distributed according to the equilibrium $F$. When a particle reaches the boundary it collides with particles from the stratum and re-enters the domain. Note that, unlike the specular reflection condition, in the diffuse reflections case the velocity of a particle after reflection is independent from its velocity before reflection. \\
The diffusion approximation of \eqref{eq:LinBoltz} is obtained by investigating the long time, small mean-free-path asymptotic behaviour of $f$. To this end, we introduce the Knudsen number $\eps$ and the following rescaling of \eqref{eq:LinBoltz}-\eqref{def:generalBC}: 
\begin{equation} \label{eq:LinBoltzrescaled} 
	\left\{ \begin{aligned} 
		& \eps^{2s} \pa_t f_\eps +  \eps v\cdot \na_x f_\eps  = L(f_\eps) &\text{ in } (0,+\infty) \times\Omega\times\RR^d \\
		& f_\eps (0,x,v) = f_{in} (x,v) &\text{ in } \Omega\times\RR^d \\
		&\gamma_- f_\eps (t,x,v) = \B_\alpha [ \gamma_+ f_\eps] (t,x, v) &\text{ on }  (0,+\infty) \times \Gamma_-
	\end{aligned} \right. 
\end{equation}
We see that the particular choice of power of $\eps$ for the rescaling in time depends on the equilibrium $F$. This is due to the fact that, for such a linear Boltzmann model as \eqref{eq:LinBoltzrescaled}, the limit diffusion process will be a $2s$-stable Levy process, with $s$ the parameter of the polynomial decay of $F$, as was proved e.g. in \cite{Mellet11,Jara2009,Mellet10,AbdallahMelletPuel} when $\Omega=\RR^d$. Our choice of rescaling \eqref{eq:LinBoltzrescaled} follows directly from the self-similar property of this Levy process, or equivalently from the fact that the fractional Laplacian of order $s$ is homogenous of degree $2s$. Note that, in general, one does not need to know a priori the power of $\eps$ that is needed in order to derive a fractional diffusion approximation. \\
In the case $\Omega=\RR^d$, it was proved in \cite{Mellet11,Mellet10,AbdallahMelletPuel} via different methods that in the limit as $\eps$ goes to $0$, $f_\eps$ converges in some weak sense to a function 
\begin{align*}
	\rho(t,x) F(v) \in \ker (L) := \left\{ \phi F,\, \phi \mbox{ independent of } v \right\}
\end{align*}
where $\rho$ is the weak solution to a fractional diffusion equation of the form
\begin{equation*}
	\left\{ \begin{aligned}
		& \pa_t \rho + \kappa \Dels \rho = 0 & \mbox{ in } (0,+\infty)\times\RR^d,\\
		&\rho(0,x) = \rho_{in} (x) = \int_{\RR^d} f_{in} \d v & \mbox{ in } \RR^d.
	\end{aligned} \right.
\end{equation*}
Recall that the fractional Laplacian $\Dels$ is a non-local integro-differential operator, infinitesimal generator of $2s$-stable Lévy process, which can be defined through its Fourier transform:
\begin{align*}
	\mathcal{F} \left( \Dels \rho \right) (\xi) := - |\xi|^{2s} \mathcal{F} \left( \rho \right) (\xi) 
\end{align*}
or equivalently as a singular integral
\begin{align*}
	\Dels \rho(x) = c_{d,s} P.V. \int_{\RR^d} \frac{\rho(x)-\rho(y)}{|x-y|^{d+2s}} \d y 
\end{align*}
where $c_{d,s}$ is an explicit constant, see e.g. \cite{DiNezzaPalatucciValdinoci12,Kwasnicki15} for more details. \\

Since our equation is set in a  subset $\Omega$ of $\RR^d$ we expect to derive a fractional diffusion equation confined to the domain $\Omega$. 
The question at the heart of this paper is to determine the appropriate boundary conditions for this asymptotic equation. When the thermodynamical equilibrium $F$ is a Gaussian (or Maxwellian) distribution it is well known that the diffusion limit of \eqref{eq:LinBoltzrescaled}, with $s=1$, leads to the classical heat equation with an homogeneous Neumann boundary condition.
Interestingly, this boundary behaviour is not very sensitive to the type of kinetic boundary conditions in the sense that if \eqref{eq:LinBoltzrescaled} is supplemented with any Maxwell boundary condition \eqref{def:MaxwellBC} with $\alpha \in [0,1]$ -- including the purely specular and purely diffuse reflection conditions -- then the limiting boundary condition is the same homogeneous Neumann condition. 

Moreover, if one considers non-linear Boltzmann models with more conservation laws (e.g. mass, momentum and energy) then one will derive fluid equations in macroscopic limits. In the acoustic regime, the limiting boundary condition is again not very sensitive to the kinetic reflection condition, as shown in \cite{JiangLevermoreMasmoudi10}. However, in hydrodynamical scalings, the boundary condition of the Stokes or Navier-Stokes limits derived in \cite{MasmoudiSaintRaymond03} and \cite{JiangMasmoudi16} does depend on the kinetic boundary interaction. More precisely, it was proved in \cite{MasmoudiSaintRaymond03} and \cite{JiangMasmoudi16} that if the accommodation coefficient $\alpha$ in \eqref{def:MaxwellBC} is fixed or goes to $0$ slower than the Knudsen number $\eps$ then one obtains a Dirichlet-type boundary condition on the limit system. On the other hand, if $\alpha$ depends on $\eps$ in such a way that $\frac{\alpha}{\eps} \rightarrow C <+\infty$ than one recovers a Navier boundary condition in the limit.

The question of boundary behaviour is very delicate with non-local operators such as the fractional Laplacian. 
Indeed, these operators are associated with $\alpha$-stable L\'evy processes (or jump processes). Unlike a Brownian motion, these processes are discontinuous and may exit the domain without touching the boundary.
This is the reason why the usual Dirichlet problem for the fractional Laplacian requires a prescribed data everywhere outside of $\Omega$ rather than just on the boundary $\pa \Omega$.
Neumann boundary value problems correspond to processes that are not allowed to jump outside $\Omega$ (sometimes referred to as censored stable processes).
Several construction of such processes are possible. 
For instance, one can cancel the process after any outside jump and restart it at its last position inside the set (resurrected processes).
This construction, see \cite{BogdanBurdzyChen03,GuanMa05,GuanMa06} for details, leads to the regional fractional Laplacian defined by
\begin{align} \label{def:regionalDels}
	\Dels_\Omega \rho(x) = c_{N,s} P.V. \int_{\Omega} \frac{\rho(x)-\rho(y)}{|x-y|^{N+2s}} \d y 
\end{align}
However, other construction of censored processes (e.g. the mirror reflection described below) are possible and will lead to different operators.
Note that, because of the non-local nature of the problem, the choice of boundary condition for the stochastic process may change the properties of its generator inside the domain and thus may lead to very different PDEs. 
Several such problems have been studied in the literature, see e.g. \cite{Barles14,GuanMa05,FelsingerKassmannVoigt15,DiPierroRosotonValdinoci17}.

In \cite{Cesbron18}, I studied the derivation of a non-local diffusion equation from a Vlasov-Lévy-Fokker-Planck model in a bounded domain with specular reflections on the boundary. It is well known that this kinetic model has the same asymptotic behaviour under a diffusive scaling as the linear Boltzmann equation which is studied in the present paper. In that case, the asymptotic 
equation reads
$$
\left\{ \begin{aligned}
	&\pa_t \rho + (-\Delta)_{\SR}^s \rho = 0 \quad & \mbox{ for } (t,x)\in (0,+\infty)\times\Omega \\
	&\rho(0,x) = \rho_{in}(x) &\mbox{ for } x\in \Omega
\end{aligned} \right. 
$$
with 
\begin{equation} \label{def:DelsSR}
	(-\Delta)_{\SR}^s\rho(x) = c_{N,s} P.V. \underset{\RR^d}{\int} \frac{\rho(x) - \rho\big(\eta(x,w)\big)}{|w|^{d+2s}} \d w
\end{equation}
where $\eta: \Omega\times\RR^d \to \bO$ is the flow of the free transport equation with specular reflection on the boundary. When $\Omega$ is the upper-half space $\lbrace y=(y',y_d)\in\RR^{d-1}\times\RR_+^*\rbrace$, we simply have
\begin{equation} \label{def:etaHS}
	\eta(x,w) = \begin{cases}
		x+w & \mbox{ if } x_d+w_d>0 \\
		(x'+w',-x_d-w_d)& \mbox{ if } x_d+w_d<0 
	\end{cases}
\end{equation}
and the underlying alpha stable process is the process which is moved back inside $\Omega$ by a mirror reflection about the boundary $\pa\Omega$ upon leaving the domain (see \cite{Cesbron18,Barles14}). \\
More recently, with A. Mellet and M. Puel, I considered in \cite{CesbronMelletPuel} the linear Boltzmann model \eqref{eq:LinBoltz} with diffusive boundary condition \eqref{def:DiffuseBC}. In that case, the asymptotic operator is
\begin{equation} \label{def:DelsN}
	(-\Delta)^s_{\mbox{\tiny N}} \rho  = - \gamma_0\int_\Omega \na \rho(y) \cdot\frac{y-x}{|x-y|^{N+2s}}\d y
\end{equation}
for some explicit $\gamma_0>0$. This operator is neither the regional fractional Laplacian, nor the operator \eqref{def:DelsSR}, and as far as we know the stochastic process it generates has yet to be constructed. Furthermore, this operator can be written in divergence form as $(-\Delta)^s_{\mbox{\tiny N}} [\rho] = \div D^{2s-1}[\rho]$ where $D^{2s-1}[\rho]$ is a non-local gradient of order $2s-1$ defined as 
\begin{equation} \label{def:Dlim}
	\Dlim [\psi] (x) = \gamma_0 \int_{w\cdot n<0} \Big( \psi(x+ w) - \psi(x) \Big) \frac{w}{|w|^{d+2s}} \d w
\end{equation}
and the non-local diffusion equation we have derived is then supplemented by the following Neumann-type condition
$$  D^{2s-1}[\rho] \cdot n=0 \qquad \mbox{ on } \pa\Omega.$$
In particular, note that while the operator $D^{2s-1}$ is non-local, the boundary condition itself is only assumed to hold on the boundary $\pa\Omega$. This is thus different from the non-local Neumann problem introduced in \cite{DiPierroRosotonValdinoci17}, where the Neumann condition is set in $\RR^d\setminus \Omega$. In \cite{CesbronMelletPuel}, we also proved well-posedness in $C^0(0,+\infty;L^2(\Omega)\cap L^2(0,+\infty; \mathcal{D}((-\Delta)^s_{\mbox{\tiny N}}) )$ of the fractional Neumann problem 
\begin{equation} \label{eq:LimPbDiffusive}
	\left\{ \begin{aligned} 
		& \pa_t \rho + (-\Delta)^s_{\mbox{\tiny N}} \rho = 0 & \mbox{ in } (0,+\infty)\times\Omega,\\
		& \Dlim [\rho](x) \cdot n = 0 & \mbox{ on } (0,+\infty)\times\dO, \\
		& \rho(0,x) = \rho_{in} (x) & \mbox{ in } \Omega
	\end{aligned} \right.
\end{equation}
for any $\rho_{in} \in L^2(\Omega)$. \\

The purpose of the present paper is to consider general Maxwell boundary condition \eqref{def:MaxwellBC} for the Linear Boltzmann model \eqref{eq:LinBoltzrescaled} in a half-space $\Omega = \RR^d_+$. To that end, we will first unify the methods developed in \cite{Cesbron18} and \cite{CesbronMelletPuel} to study diffusion limits respectively for specular and diffuse reflections.

\subsection{Main results and outline of the paper}
The existence of solutions to \eqref{eq:LinBoltz} with boundary condition \eqref{def:MaxwellBC} is a delicate problem because it is difficult to control the trace $\gamma_+ f$ in an appropriate functional space, see e.g. \cite{Mischler10}.
Note that for a  given test function  $\phi\in\mathcal D ([0,\infty)\times\overline\Omega\times\RR^d)$,  
smooth solutions of \eqref{eq:LinBoltzrescaled}-\eqref{def:MaxwellBC} with $\alpha \in [0,1]$ satisfies
\begin{align*}
	& - \underset{\RR^+\times\Omega\times\RR^d}{\iiint} f^\eps \Big(  \pa_t \phi +\eps^{1-2s} v \cdot \na_x\phi \Big) \d v \d x \d t \\
	&+ \eps^{1-2s} \underset{\RR^+\times \Gamma_+}{\iint} \gamma_+ f^\eps  \bigg( \gamma_+ \phi - \mathcal B ^*_\alpha[\gamma_- \phi] \bigg) |v\cdot n|\d v\d S(x) \d t\\
	& =\eps^{-2s}\underset{\RR^+\times\Omega\times\RR^d}{\iiint} f^\eps L^*(\phi) \d v\d x\d t+ \underset{\Omega\times\RR}{\iint} f_{in} (x,v) \phi(0,x,v) \d  x \d v.
\end{align*}
with 
\begin{align*}
	L^* (\phi) (t,x,v) = \nu_0 \bigg( \int_{\RR^d} \phi(t,x,w) F(w) \d w - \phi(t,x,v) \bigg) 
\end{align*}
and for any $(x,v)\in\Gamma_+$
\begin{align} \label{def:Bdiffstar}
	\mathcal B ^*_\alpha [\gamma_- \phi ] (x,v)= (1-\alpha)\mathcal{B}^*_{\SR} [\gamma_- \phi] (x,v) + \alpha \mathcal{B}^*_{\D} [\gamma_- \phi] (x,v)
\end{align}
where 
\begin{equation} \label{eq:Bstar}
	\left\{ \begin{aligned} 
		& \mathcal{B}^*_{\SR} [\gamma_- \phi] (x,v) = \gamma_- \phi\big( t,x, v - 2(v\cdot n(x)) n(x) \big),\\
		& \mathcal{B}^*_{\D} [\gamma_- \phi] (x,v) = c_0  \int_{w\cdot n(x)<0} \gamma_-\phi(t,x,w) F(w) |w\cdot n(x)|\d w.
	\end{aligned} \right.
\end{equation}
A classical way of defining weak solutions of \eqref{eq:LinBoltzrescaled}-\eqref{def:MaxwellBC} without having to deal with the trace $\gamma f$ is then the following:
\begin{defi}\label{def:weaksolKin}
	We say that a function $f(t,x,v)$ in $L^2_{F^{-1}}((0,\infty)\times \Omega\times\RR^d)$ is a weak solution of \eqref{eq:LinBoltzrescaled}-\eqref{def:MaxwellBC} if for any test functions  $\phi(t,x,v)$ such that 
	$\phi$, $\pa_t \phi$ and $v\cdot \na_x \phi$ are $L^2_{F}((0,\infty)\times \Omega\times\RR^d)$
	and which satisfies the boundary condition
	$$\gamma_+ \phi = \mathcal B_\alpha^*[\gamma_- \phi] $$
	the following equality holds:
	\begin{align}
		&- \underset{\RR^+\times\Omega\times\RR^d}{\iiint} f^\eps \Big(  \pa_t \phi +\eps^{1-2s} v \cdot \na_x\phi \Big) \d v\d x\d t \nonumber \\
		& \qquad\quad =\eps^{-2s}\underset{\RR^+\times\Omega\times\RR^d}{\iiint} f^\eps L^*(\phi) \d v\d x\d t + \underset{\Omega\times\RR^d}{\iint} f_{in} (x,v) \phi(0,x,v) \d x\d v.\label{eq:weak0}
	\end{align}
\end{defi}
Here and in the rest of the paper, we used the notation
$$ L^2_{F^{-1}} ((0,\infty)\times \Omega\times\RR^d)= \left\{f(t,x,v) \, ;\, \int_0^\infty\int_\Omega\int_{\RR^d} |f(t,x,v)|^2 \frac{1}{F(v)}\, dv\, dx\, dt<\infty\right\}$$
and a similar definition for $L^2_{F}((0,\infty)\times \Omega\times\RR^d)$.

Our first theorem concern the specular reflection case for which the method we develop in this paper is particularly efficient. This is the only case in this paper where we will consider convex domains and not just half-spaces. We will give a precise characterisation of admissible domains in Section \ref{subsec:FT}, note in particular that the result holds for the unit ball in $\RR^d$. In order to state our result, let us define the operator $\L_{\SR}$ as 
\begin{equation}\label{def:LSR}
	\L_{\SR} [\psi](x) = -\gamma_1 P.V. \int_{\RR^d} \frac{ \psi(x) - \psi\big( \eta(x,v)\big) }{|v|^{d+2s}} \d  v
\end{equation}
where $\eta$ the flow of free transport with specular reflections, see Section \ref{subsec:FT} for details, and the constant $\gamma_1$ is given by
\begin{equation}\label{def:gamma1}
	\gamma_1 = \gamma \nu_0^{1-2s} \Gamma(2s+1)
\end{equation}
with $\gamma$ the constant of the velocity equilibrium $F$, see \eqref{eq:F0}. Note that the operator $\L_{\SR}$ is equal, up to a negative constant, to $\DelsSR$ introduced in \cite{Cesbron18} and restated above in \eqref{def:DelsSR}. We have changed the constant in an effort to homogenise the notations of this paper. Furthermore, we also recall the definition of the functional space $\HSRs(\Omega)$ introduced in \cite{Cesbron18}:
\begin{equation} \label{def:HSRs}
	\HSRs (\Omega) = \Big\{ \psi\in L^2(\Omega): \, \iint_{\Omega\times\RR^d} \Big( \psi(x) - \psi\big(\eta(x,v)\big)\Big)^2\frac{1}{|v|^{d+2s}} \d v \d x <\infty \Big\} .
\end{equation}
Our first result reads as follows
\begin{thm} \label{thm:SR}
	Assume $F$ satisfies \eqref{eq:F0} with $s\in(0,1)$ and let $\Omega$ be an admissible domain in the sense of Definition \ref{def:admissibleOmega}. Assume that $f_\eps$ is a weak solution of \eqref{eq:LinBoltzrescaled}-\eqref{def:SpecularBC} in $\RR_+\times\Omega\times\RR^d$ in the sense of Definition \ref{def:weaksolKin}. Then $f_\eps$ converges weakly in $L^\infty(0,+\infty; L^2_{F^{-1}}(\Omega\times\RR^d))$, as $\eps$ goes to $0$, to the function $\rho(t,x)F(v)$ where $\rho$ is the unique weak solution in $C^0(0,+\infty;L^2(\Omega))\cap L^2(0,+\infty,\HSRs (\Omega))$ to
	\begin{equation} \label{eq:LimpbSR}
		\left\{ \begin{aligned}
			&\pa_t \rho - \L_{\SR} [\rho] = 0  & \mbox{ in }\RR_+\times\Omega,\\
			&\rho(0,x) = \rho_{in}(x) & \mbox { in }\Omega.
		\end{aligned} \right.
	\end{equation}
\end{thm}
As expected, this asymptotic behaviour is the same as the one established in \cite{Cesbron18} for a Vlasov-Lévy-Foker-Planck model. \\

Our second theorem concerns the diffusive boundary condition. The theorem itself is exactly the same as the main result of \cite{CesbronMelletPuel} although the proof will be different and, in particular, it leads us to define the limit operator $\L_{\D}$ with the following decomposition 
\begin{equation} \label{def:LDlim} 
	\L_{\D}[\psi] = - \gamma_{d,s}\Delsreg[\psi] - \kappa[\psi]
\end{equation}
where $\Delsreg$ is defined in \eqref{def:regionalDels}, the constant $\gamma_{d.s}$ is given by $\gamma_{d,s} := \frac{\gamma_1}{c_{d,s}}$ and 
\begin{equation} \label{def:kappalim}
	\kappa[\psi](x) :=   \pv \int_{x+v\notin\Omega} \Big( \psi(x) - \psi(x_f)\Big) \frac{ \gamma_1 }{|v|^{d+2s}} \d v
\end{equation}
with $x_f=x_f(x,v)$ the forward exit point:
\begin{equation} \label{def:xftf}
	\left\{ \begin{aligned}
		& x_f(x,v) = x + \tau_f(x,v) v \in \dO \\
		& \tau_f (x,v) = \inf \lbrace \tau>0 : \, x+\tau v \notin \Omega \rbrace .
	\end{aligned} \right.
\end{equation}
Nevertheless, one can easily check, using results from \cite{CesbronMelletPuel}, that $\L_{\D} = -(-\Delta)^s_{\mbox{\tiny N}}$ expressed in \eqref{def:DelsN} above and so we still have $\L_{\D} = \div \Dlim$ with $\Dlim$ given by \eqref{def:Dlim} and with $\gamma_0 = \gamma \nu_0^{1-2s} \Gamma(2s)$. Moreover, noticing that $\na_v \cdot\left(\frac{v}{|v|^{d+2s}}\right)  = -\frac{2s}{|v|^{d+2s}}$ and recalling the fact that $v\cdot \na_v x_f(x,v) = 0$ we get with integration by parts  
\begin{equation} \label{def:kappalim2}
	\kappa[\psi](x) =  \gamma_0 \pv \int_{\dO} \Big( \psi(x)-\psi(y)\Big) \frac{ (x-y)\cdot n(y) }{|y-x|^{d+2s}} \d \sigma(y) .
\end{equation}
which is also a corollary of \cite[Lemma 2.5]{CesbronMelletPuel}. The theorem reads
\begin{thm}\label{thm:Diff}
	Assume $F$ satisfies \eqref{eq:F0} with $s\in(1/2,1)$ and let $\Omega$ be the half-space $\RR^d_+$. Assume that $f_\eps$ is a weak solution of \eqref{eq:LinBoltzrescaled}-\eqref{def:DiffuseBC} in $\RR_+\times\Omega\times\RR^d$ in the sense of Definition \ref{def:weaksolKin}. Then $f_\eps$ converges weakly in $L^\infty(0,+\infty; L^2_{F^{-1}}(\Omega\times\RR^d))$, as $\eps$ goes to $0$, to the function $\rho(t,x)F(v)$ where $\rho$ satisfies:
	\begin{equation} \label{eq:LimpbDif}
		\iint_{\RR_+\times\Omega} \rho(t,x) \Big( \pa_t \psi(t,x) + \L_{\D}[\psi] \Big) \d t \d x + \int_\Omega \rho_{in} (x)\psi(0,x)\d x = 0 
	\end{equation}
	for all test function $\psi\in W^{1,\infty}(0,+\infty;H^2(\Omega))$ such that $\L_{\D} [\psi] \in L^2(\RR_+\times\Omega)$ and 
	\begin{equation} \label{eq:BDLimpbDiff}
		\Dlim [\psi](t,x) \cdot n(x) = 0  \qquad (t,x) \in \RR_+\times\dO. 
	\end{equation}
\end{thm}

Finally, our third and fourth theorems concern the Maxwell boundary conditions with accommodation coefficient $\alpha \in (0,1)$ in a half-space $\Omega= \RR^{d-1}\times\RR_+^*$. The first concerns the fractional diffusion limit and reads
\begin{thm}\label{thm:Maxwell}
	Assume $F$ satisfies \eqref{eq:F0} with $s\in(1/2,1)$ and let $\Omega$ be the half-space $\RR^{d-1}\times\RR_+^*$. Assume that $f_\eps$ is a weak solution of \eqref{eq:LinBoltzrescaled}-\eqref{def:MaxwellBC} in $\RR_+\times\Omega\times\RR^d$ in the sense of Definition \ref{def:weaksolKin}. Then $f_\eps$ converges weakly in $L^\infty(0,+\infty; L^2_{F^{-1}}(\Omega\times\RR^d))$, as $\eps$ goes to $0$, to the function $\rho(t,x)F(v)$ where $\rho$ satisfies:
	\begin{equation} \label{eq:LimpbMax}
		\iint_{\RR_+\times\Omega} \rho(t,x) \Big( \pa_t \psi(t,x) +  (1-\alpha) \L_{\SR}[\psi] + \alpha \L_{\D}[\psi] \Big) \d t \d x + \int_\Omega \rho_{in} (x)\psi(0,x)\d x = 0 
	\end{equation}
	for all test function $\psi\in W^{1,\infty}(0,+\infty;H^2(\Omega))$ such that $ (1-\alpha) \L_{\SR} [\psi]+ \alpha \L_{\D}[\psi] \in L^2(\RR_+\times\Omega)$ and 
	\begin{equation} \label{eq:BDLimpbMax}
		\Dlim [\psi](t,x) \cdot n(x) = 0  \qquad (t,x) \in \RR_+\times\dO. 
	\end{equation}
\end{thm}
This theorem will come as a corollary of Theorem \ref{thm:SR} and Theorem \ref{thm:Diff}. We conclude this paper with a well-posedness result for this new limit problem. Introducing the notation 
\begin{equation}\label{def:LM}
	\L_{\M} [\psi] = (1-\alpha) \L_{\SR}[\psi] + \alpha \L_{\D}[\psi]
\end{equation}
for any fixed $\alpha \in [0,1]$ and $s\in(1/2,1)$, we have
\begin{thm} \label{thm:MaxwellSWP}
	For all $\rho_{in} \in L^2(\Omega)$, the evolution problem
	\begin{equation} \label{eq:LimpbMaxwell} 
		\left\{ \begin{aligned}
			& \pa_t \rho - \L_{\M} [\rho] = 0 &\mbox{ in } (0,+\infty)\times\Omega \\
			&\alpha\Dlim [\rho](x)\cdot n(x) = 0 & \mbox{ on } (0,+\infty)\times\dO \\
			&\rho(0,x) = \rho_{in} (x) & \mbox{ in } \Omega .
		\end{aligned} \right.
	\end{equation}
	has a unique weak solution $\rho \in \mathcal{C}^0(0,+\infty; L^2(\Omega)) \cap L^2(0,+\infty; \mathcal{D}(\L_{\M}))$ with 
	\begin{equation} \label{def:DLM}
		\mathcal{D}(\L_{\M}):= \Big\{ \psi\in H^s(\Omega) ; \, \L_{\M}[\phi]\in L^2(\Omega) \quad \mbox{and}\quad  \alpha\Dlim[\psi]\cdot n =0 \mbox{ on } \dO \Big\}. 
	\end{equation}
\end{thm}
Note that this last theorem naturally includes the well-posedness results established in \cite{Cesbron18} and \cite{CesbronMelletPuel} for the cases $\alpha=0$ and $\alpha=1$ respectively. This motivates the coefficient $\alpha$ in the boundary condition.  \\
We make some remarks about these results:

\begin{enumerate}
	
	\item There is a significant difference between Theorem \ref{thm:SR} and the Theorems \ref{thm:Diff} and \ref{thm:Maxwell}. In the specular case, we are able to identify the weak limit $\rho$ of the kinetic solution $f_\eps$ with the unique weak solution to \eqref{eq:LimpbSR} in $C^0(0,+\infty;L^2(\Omega)) \cap L^2(0,+\infty; \HSRs(\Omega))$ thanks to a detailed analysis of the boundary behaviour of the solution to this problem, we refer the interested reader to \cite[Section 5.3]{Cesbron18}. However, similar identifications are not yet available in the diffusive case, and consequently in the Maxwell case. In both those cases, we have proved that the weak limit $\rho$ of (a subsequence of) $f_\eps$ is solution to the limit problem in the sense stated in the theorems, and we have also proved uniqueness of weak solutions as stated for the Maxwell case in Theorem \ref{thm:MaxwellSWP}. However, we cannot identify these solutions yet. So far, in the diffuse reflections case, we have only been able to do this identification in an interval in dimension 1 in \cite{CesbronMelletPuel2} by adding stronger conditions on the test functions in \eqref{eq:LimpbDif} to the price of an additional assumption on the initial data $\fin$. 
%
	\item Using the following integration by parts formulae proved in \cite{Cesbron18,CesbronMelletPuel}
	\begin{align*}
		&\int_\Omega \psi \L_{\SR} [\phi] \d x  = \int_{\Omega} \phi \L_{\SR} [\psi] \d x, \\
		&\int_{\Omega} \psi \L_{\D}[\phi] \d x - \int_{\Omega} \phi \L_{\D} [\psi] \d x = \int_{\dO} \big[ \psi \Dlim [\phi] \cdot n - \phi \Dlim [\psi] \cdot n \big] \d \sigma(x) 
	\end{align*}
	we see that, assuming these formulae hold for $\psi$ and $\phi$ in $\mathcal{D}(\L_{\M})$, the equation \eqref{eq:LimpbMax} can be seen as a weak formulation of the fractional Neumann boundary problem \eqref{eq:LimpbMaxwell}.
	\item We would like to emphasise the fact that, although in this paper the results for the Maxwell boundary conditions appear to be a sum of the phenomena from the pure specular and pure diffusive cases, it is because we are in the half-space and we do not expect to have such a simple interaction between the specular and diffusive conditions in a more general convex domain. Morally, the half-space is a very particular case because the trajectories associated with the free-transport part of the kinetic model interact at most once with the boundary. 
	
\end{enumerate}

\paragraph{Outline of the paper:} The paper is organised as follows: in Section 2 we recall some useful results about the kinetic equation \eqref{eq:LinBoltzrescaled} and the free transport equation in order to prove Theorem \ref{thm:SR} and Theorem \ref{thm:Diff}. In Section 3 we will see that Theorem \ref{thm:Maxwell} comes as a corollary of the previous two theorems and then focus on the proof of well-posedness of \eqref{eq:LimpbMaxwell}, i.e. Theorem \ref{thm:MaxwellSWP}. We conclude this paper with Appendix \ref{subsection:CVPhieps} where we prove a lemma of convergence of test functions which is useful in the previous three sections, we chose to prove this lemma in a independent section to avoid repetitions and to emphasise on the convergence of operators in the other sections. 

\section{The specular and the diffuse reflections boundary conditions}

\subsection{Preliminary results}

\subsubsection{A priori estimates}

Let us recall the following classical result which shows the convergences of $f_\eps$, solution to \eqref{eq:LinBoltzrescaled}-\eqref{def:MaxwellBC}, toward the thermodynamical equilibrium, i.e. the kernel of $L$: 
\begin{lemma} \label{lem:Apriori}
	Let $f_{in}$ be in $L^2_{F^{-1}}(\Omega\times\RR^N)$. The weak solution $f^\eps$ of \eqref{eq:LinBoltzrescaled} with boundary condition \eqref{def:MaxwellBC} satisfies, up to a subsequence
	$$
	f^\eps \rightarrow \rho(t,x) F(v) \quad \text{ weakly in } L^\infty(0,+\infty; L^2_{F^{-1}}(\Omega\times\RR^N))
	$$
	where $\rho(t,x)$ is the weak limit of $\rho^\eps (t,x) = \int_{\RR^N} f^\eps \d v $ and, moreover, 
	$$
	\lVert f^\eps - \rho_\eps F \lVert_{L^2_{F^{-1}}(\Omega\times\RR^N)} \rightarrow 0 \quad \text{ as } \eps \rightarrow 0. 
	$$
\end{lemma} 
This lemma is proved using equation \eqref{eq:LinBoltzrescaled} to control the weighted $L^2_F$-norm of the solution $f_\eps$, and formal estimates of the trace that follow from the boundary conditions. In an effort of concision we will not repeat this proof here and refer the interested reader e.g. to \cite[Proposition 1.2]{Cesbron18} and \cite[Lemma 2.1]{CesbronMelletPuel} for the proofs in the cases $\alpha=0$ and $\alpha=1$ respectively, the proof for any $\alpha \in (0,1)$ is a direct corollary of those two particular cases.

\subsubsection{The free transport equation} \label{subsec:FT}

In this section we consider the free transport equation in a bounded domain $\Omega$ with specular reflections on the boundary and initial data uniform in velocity
\begin{equation} \label{eq:FreeTransport}
	\left\{ \begin{aligned}
		& \pa_t f + v\cdot \na_x f = 0 & \mbox{ in } \RR_+\times\Omega\times\RR^d \\
		& f(0,x,v) = f_{in} (x) &\mbox{ in } \Omega\times\RR^d\\
		&\gamma_- f(t,x,v) = \gamma_+ f\big(t,x, v-2(v\cdot n(x)) n(x) \big) & \mbox{ on } \RR_+\times\Gamma_-.
	\end{aligned} \right.
\end{equation}
This equation will play a crucial role in the study of the asymptotic behaviour of Linear Boltzmann equation with specular reflections in Section \ref{sec:SR}. In particular, we are interested in the propagation of Sobolev regularity and our requirement for such propagation will give rise to our definition of admissible domain $\Omega$ for which Theorem \ref{thm:SR} holds. Although the general propagation of Sobolev (or H\"older) regularity for this transport equation is still an open question, we do have some results on the regularity of the spatial flow which are sufficient in the context of fractional diffusion limits and which we shall recall now. \\
The characteristic equation associated with \eqref{eq:FreeTransport} reads
\begin{equation*}
	\left\{ \begin{aligned}
		& \dot{X}_t = V_t, & X_0= x, \\
		& \dot{V}_t = 0, & V_0 = v,\\
		& V_{t^+} = \mathcal{R}_{X_t} (V_{t^-}) & \mbox{ for all } t \mbox{ such that } X_t \in \dO. 
	\end{aligned}\right.
\end{equation*}
with $R_y(w) = w - 2 (w\cdot n(y)) n(y)$ is the specular reflection operator, with $n(y)$ the outward normal vector at $y\in\dO$. \\
We denote $F_t$ the  flow of our transport problem: for all $(t,x,v) \in \RR_+\times\Omega\times\RR^d$ we have $\mathcal{F}_t(x,v) := (X_t(x,v), V_t(x,v))$. We then have the following existence result from \cite{Halpern}
\begin{thm}[Theorem 3, \cite{Halpern}]  \label{thm:etaexist}
	Let us call $\zeta$ the function such that 
	$$ \Omega = \lbrace x\in\RR^d / \zeta(x)<0 \rbrace \hspace{0.2cm} \text{and} \hspace{0.2cm} \dO =\lbrace x\in\RR^d / \zeta(x)=0 \rbrace.$$
	If $\zeta$ has a bounded third derivative and nowhere vanishing curvature in the sense that there exists a constant $C_\zeta >0$ such that for all $\xi\in\RR^d$:
	\begin{align} \label{def:convexityFlow}
		\underset{i,j=1}{\overset{d}{\sum}} \xi_i \frac{\pa^2 \zeta}{\pa x_i \pa x_j} \xi_j \geq C_\zeta |\xi|^2
	\end{align} 
	then $\mathcal{F}_t(x,v)$ is well defined for all $(x,v)\in\Omega\times\RR^d$.
\end{thm}
Morally, the convexity assumption \eqref{def:convexityFlow} ensures that the grazing trajectories stay confined to the grazing set and do not transport singularities inside the domain. \\
Now that we have existence, we restrict our investigation to the spatial flow at time $t=1$, namely we define a function $\eta$ as, for all $(x,v)\in\Omega\times\RR^d$
\begin{equation}\label{def:eta}
	\eta(x,v) = X_{t=1}(x,v). 
\end{equation}
This function $\eta$ was introduced and studied in \cite{Cesbron18}. When $\Omega$ is a half-space, $\eta$ has a simple explicit expression mentioned above in \eqref{def:etaHS}. When $\Omega$ is a ball, one can also express $\eta$ rather explicitly, we refer the interested reader to \cite[Appendix A]{Cesbron18} for more details. The function $\eta$ plays a crucial role in the fractional diffusion limit of linear kinetic equations with specular reflections, to the point that it become an integral part of the limit operator $\L_{\SR}$ as one can see in \eqref{def:LSR}. As a consequence, it is not surprising that some regularity of $\eta$ is required in order to pass to the limit in the kinetic model. This regularity is entirely dependent on the domain $\Omega$ so that it can be seen as an assumption on the domain itself. This leads us to the following definition of admissible domains. 
\begin{defi} \label{def:admissibleOmega}
	We say that a bounded open domain $\Omega\subset \RR^d$ is admissible if 
	\begin{itemize}
		\item for all $\phi\in C^\infty_c(\bO)$ such that $\pa_n \phi = 0 $ on $\dO$ :
		\begin{equation}  \label{eq:admissibleSobolev}
			\big| D_v^2 [\phi \circ \eta] (x,v)  \big|\in L^2_{\mu}(\Omega\times V) 
		\end{equation}
		where $V\subset \RR^d$, $\mu$ is a radial measure such that $\mu(V)<\infty$ and $|\cdot |$ is any matrix norm,
		\item the map 
		\begin{equation}  \label{eq:admissibleJacobian}
			(x,v)\in \Omega\times\RR^d \mapsto \mathcal{F}_{t=1}(x,v) =\big(\eta(x,v), (v\cdot\na_x)\eta(x,v)\big)
		\end{equation} 
		has a unitary Jacobian determinant 
	\end{itemize}
\end{defi}
These conditions are fulfilled if $\Omega$ is a ball in $\RR^d$, we refer the interested reader to \cite[Lemma A.3, Lemma 5.4]{Cesbron18}. Moreover, note that the second assumption \eqref{eq:admissibleJacobian} should not be absolutely necessary, as long as the Jacobian determinant is finite and never cancels one should be able to adapt our method and derive similar results.

\subsection{The specular reflections case}  \label{sec:SR}

We consider the rescaled Linear Boltzmann equation with specular reflections boundary condition \eqref{eq:LinBoltzrescaled}-\eqref{def:SpecularBC} on an admissible spatial domain $\Omega$ in the sense of Definition \ref{def:admissibleOmega}. \\
Given a test function $\phi$ which satisfies $\gamma_+ \phi = \BSRs [\gamma_- \phi]$ on $\Gamma_+$, the weak solution $f_\eps$ of \eqref{eq:LinBoltzrescaled} in the sense of Definition \ref{def:weaksolKin} satisfies, with $Q=(0,+\infty)\times\Omega\times\RR^d$
\begin{equation} \label{eq:wfLRSR1}
	\begin{aligned}
		&\iiint_{Q} f_\eps \pa_t \phi \d t \d x \d v + \iint_{\Omega\times\RR^d} \fin(x,v) \phi (0,x,v) \d x \d v \\
		&= - \eps^{-2s} \iiint_{Q} \Big[ f_\eps \Big( \eps v\cdot \na_x \phi - \nu_0 \phi \Big) + \nu_0 \rho_\eps F(v) \phi \Big] \d t \d x \d v .
	\end{aligned}
\end{equation}
We introduce the operator $A_\eps$ defined as 
\begin{align} \label{def:AepsSR}
	A_\eps =  \eps v\cdot \na_x -\nu_0 \text{Id} 
\end{align}
on the domain 
\begin{align} \label{def:DAepsSR} 
	\mathcal{D}(A_\eps) = \lbrace \phi \in L^2_{F}(\Omega\times\RR^d): \, v\cdot \na_x \phi \in  L^2_{F}(\Omega\times\RR^d) \mbox{ and } \gamma_+ \phi = \BSRs [\gamma_- \phi] \mbox{ on } \Gamma_+ \rbrace .
\end{align}
We then have the following proposition concerning the inverse of $A_\eps$:
\begin{prop} \label{prop:solAP}
	Given $\psi\in \mathcal{D}\big( [0,+\infty)\times \bO)$, the function $\phi_\eps:= A_\eps^{-1} [- \nu_0\psi]$ can be expressed, using $\eta : \Omega\times\RR^d \mapsto \bO$ defined in \eqref{def:eta}, as
	\begin{equation}
		\phi_\eps(t,x,v) = \int_0^{+\infty} e^{-\nu_0 \tau} \nu_0 \psi\big(t, \eta(x, \eps \tau v) \big) \d \tau .
	\end{equation}
\end{prop}
\begin{proof}
	In this proof we omit the variable $t$ which is just a parameter since $A_\eps$ does not act on $t$. It is easy to check that the boundary condition in \eqref{def:DAepsSR} is satisfied by $\phi_\eps$: for any $(x,v)\in\dO\times\RR^d$ we have $\eta(x,v) = \eta(x, \mathcal{R}_x v)$ so that 
	\begin{align*}
		\phi_\eps \big(x,\mathcal{R}_x v \big) &= \int_0^{+\infty} e^{-\nu_0 \tau} \nu_0 \psi\big( \eta(x, \eps \tau\mathcal{R}_x v) \big) \d \tau \\
		&= \int_0^{+\infty} e^{-\nu_0 \tau} \nu_0 \psi\big( \eta(x, \eps \tau v) \big) \d \tau  \\
		&= \phi_\eps (x,v) .
	\end{align*}
	To prove that $\phi_\eps= A_\eps^{-1} [- \nu_0\psi]$ we first notice that for fixed $(x,v) \in \Omega\times\RR^d$ we have
	\begin{align*}
		\frac{\d }{\d \tau} \Big[ e^{-\nu_0 \tau} \nu_0 \psi\big( \eta(x, \eps \tau v) \big)  \Big] &= - \nu_0 e^{-\nu_0 \tau} \nu_0 \psi\big( \eta(x, \eps \tau v) \big) \\
		&\quad + e^{-\nu_0\tau} \nu_0\eps v \cdot \na_v \eta(x, \eps \tau v) \cdot \na \psi\big( \eta(x, \eps\tau v) \big)
	\end{align*}
	where $v\cdot \na_v \eta = v\cdot \na_x \eta$ by construction of $\eta$ and, moreover we recognise 
	$$\eps v\cdot \na_x \eta(x,\eps \tau v) \na \psi \big( \eta(x, \eps\tau v) \big)  = \eps v\cdot \na_x \big[  \psi \big( \eta(x, \eps \tau v) \big) \big]$$
	hence
	\begin{align*}
		\frac{\d }{\d \tau} \Big[ e^{-\nu_0 \tau} \nu_0 \psi\big( \eta(x, \eps \tau v) \big)  \Big]  &=  \big( \eps v\cdot \na_x -\nu_0\big) \Big[ e^{-\nu_0 \tau} \nu_0\psi\big( \eta(x, \eps \tau v) \big)  \Big] \\
		&= A_\eps  \Big[ e^{-\nu_0 \tau} \nu_0\psi\big( \eta(x, \eps \tau v) \big)  \Big].
	\end{align*}
	A simple integration by parts in $\tau$ concludes the proof since $\psi\big( \eta(x, \eps \tau v) \big)\lvert_{\tau=0} = \psi(x)$, namely:
	\begin{align*}
		A_\eps  \phi_\eps (x,v) &= \int_0^{+\infty} A_\eps \Big[ e^{-\nu_0 \tau} \nu_0\psi\big( \eta(x, \eps \tau v) \big)  \Big] \d \tau\\
		&= \int_0^{+\infty}\frac{\d }{\d \tau} \Big[ e^{-\nu_0 \tau} \nu_0 \psi\big( \eta(x, \eps \tau v) \big)  \Big]  \d \tau \\
		&= - \nu_0 \psi(x) .
	\end{align*}
\end{proof}
The weak formulation of equation \eqref{eq:LinBoltzrescaled} with a test function $\phi_\eps = A_\eps^{-1}[-\nu_0 \psi]$ for a given $\psi\in\mathcal{D}([0,T)\times\bO)$ then reads, since $F$ is normalised:
\begin{align} \label{eq:wfLRSR}
	&\iiint_{Q} f_\eps \pa_t \phi_\eps \d t \d x \d v + \iint_{\Omega\times\RR^d} \fin(x,v) \phi_\eps (0,x,v) \d x \d v  =  - \iint_{\RR_+\times\Omega} \rho_\eps \L^\eps[\psi] \d x \d t 
\end{align}
where $\L^\eps$ is defined as
\begin{align} 
	\mathcal{L}^\eps [\psi] (t,x) &:= \eps^{-2s} \nu_0 \int_{\RR^d} \Big( \phi_\eps (t,x,v)-\psi(t,x)  \Big) F(v) \d v \nonumber \\
	&= \eps^{-2s} \nu_0 \int_{\RR^d} \int_0^{+\infty} e^{-\nu_0\tau} \nu_0 \Big( \psi\big(\eta(x,\eps \tau v) \big)-\psi(x)  \Big)F(v) \d \tau \d v \nonumber \\
	&=\eps^{-2s} \int_0^{+\infty} \int_{\RR^d} e^{-\nu_0 \tau} \nu_0^2 \Big( \psi\big(\eta(x, \eps w) \big) - \psi(x) \Big)  \tau^{-d} F\left( \frac{w}{ \tau} \right)\d w \d \tau \nonumber \\
	&= \eps^{-2s} \int_{\RR^d}  \Big( \psi\big(\eta(x,\eps w) \big)-\psi(x)  \Big) F_1(w)  \d v .  \label{def:LepsSR}
\end{align}
where we used the substitution $w=\tau v$ and $F_1$ is defined as 
\begin{equation} \label{eq:F1}
	F_1(w) := \int_0^{+\infty} e^{-\nu_0 \tau} \nu_0^2 \tau^{-d} F\left( \frac{w}{ \tau} \right) \d \tau.
\end{equation}
\subsubsection{Macroscopic limit, proof of Theorem \ref{thm:SR}}

To establish the fractional diffusion approximation we want to take the limit in this weak formulation. The convergence of the terms in the left-hand-side of the weak formulation follow from the convergence of $\phi_\eps$ to $\psi$ which we will show in a more general setting in Appendix \ref{subsection:CVPhieps} in order to avoid repetitions. \\
The main difficulty in passing to the limit in \eqref{eq:wfLRSR} lies therefore in the convergence of the operator $\mathcal{L}^\eps$. We introduce the set $\mathfrak{D}^s$ defined as 
\begin{equation} \label{def:D}
	\mathfrak{D}^s = \Big\{ \psi\in \mathcal{D}([0,+\infty)\times\bO) \mbox{ such that if } s\geq 1/2 \mbox{ then } \na_x \psi(t,x) \cdot n(x) = 0 \mbox{ for all } x\in\dO \Big\}
\end{equation}
and we have the following convergence result
\begin{prop} \label{prop:CVop}
	For any $\psi\in\mathfrak{D}^s$ we have
	\begin{align*}
		&\mathcal{L}^\eps [\psi]  \underset{\eps \rightarrow 0}{\longrightarrow} \L_{\SR} [\psi] \text{ strongly in } L^2((0,+\infty)\times\Omega) 
	\end{align*}
	with $\L_{\SR}$ defined in \eqref{def:LSR}.
\end{prop}
\begin{proof}
	We first notice that we have
	\begin{align*}
		\L_{\SR} [\psi](x) &=  \gamma_1 \text{P.V.}\int_{\RR^d} \big[ \psi \big( \eta(x, v) \big) -\psi(x)  \big] \frac{1}{|v|^{d+2s}} \d v \\
		&=  \eps^{-2s} \text{P.V.}\int_{\RR^d}  \big[ \psi \big( \eta(x, \eps w) \big)  - \psi(x) \big] \frac{\gamma_1}{|w|^{d+2s}} \d w .
	\end{align*}
	We can then write 
	\begin{align*}
		\mathcal{L}^\eps [\psi](x) -\L_{\SR} [\psi] (x) = \eps^{-2s}\text{P.V.} \int_{\RR^d} \big[ \psi\big(\eta(x,\eps w) \big) - \psi(x) \big]  G(w) \d w 
	\end{align*}
	with $G(w) = F_1(w)-\frac{\gamma_1}{|w|^{d+2s}}$ which behaves as: \\
	\begin{lemma} \label{lem:F1}
		If $F$ satisfies \eqref{eq:F0}, then $G(w)= F_1(w)- \frac{\gamma_1}{|w|^{d+2s}}$ with $F_1$ is defined in \eqref{eq:F1} satisfies
		\begin{align}
			\mbox{for all } |w|\leq 1, \, |G(w)| \lesssim \frac{1}{|w|^{d+2s}} \quad \mbox{ and for all } |w|> 1,\,  |G(w)| \lesssim  \frac{1}{|w|^{d+4s}} \label{eq:F1largev} 
		\end{align}
	\end{lemma}
	Here and throughout the paper "$\lesssim$" means "lesser than, up to a constant". 
	\begin{proof}
		We start by noticing that 
		\[ \gamma \int_0^{+\infty} e^{-\nu_0 \tau} \nu_0^2 \tau^{2s} \d \tau = \gamma \nu_0^{1-2s} \Gamma(2s+1)= \gamma_1  \]
		hence 
		\begin{align*}
			G(w) = \int_0^{\infty} e^{-\nu_0 \tau} \nu_0^2 \bigg( \tau^{-d} F\left(\frac{w}{\tau}\right) - \frac{\gamma \tau^{2s}}{|w|^{d+2s}} \bigg) \d \tau .
		\end{align*}
		For $|w|\leq 1$, we estimate using \eqref{eq:F0} 
		\begin{align*}
			|G(w)| &\leq \int_0^{|w|} e^{-\nu_0 \tau} \nu_0^2 \tau^{-d} \bigg( F\left(\frac{w}{\tau}\right) - \frac{\gamma \tau^{d+2s}}{|w|^{d+2s}}\bigg) \d \tau \\
			&\quad + \frac{1}{|w|^{d+2s}} \int_{|w|}^{+\infty} e^{-\nu_0 \tau} \nu_0^2 \tau^{2s} \Big(  \frac{|w|^{d+2s}}{\tau^{d+2s}} F\left(\frac{w}{\tau}\right) - \gamma \Big) \d \tau\\
			&\lesssim \int_0^{|w|} e^{-\nu_0\tau} \nu_0^2 |w|^{4s} \frac{1}{|w|^{d+4s}} \d \tau \\
			&\quad + \frac{1}{|w|^{d+2s}} \int_0^{+\infty} e^{-\nu_0 \tau} \nu_0^2 \tau^{2s} \Big| \lVert F\lVert_{L^\infty} - \gamma \Big| \d \tau \\
			&\lesssim \frac{\nu_0^2}{|w|^{d-1}} + \frac{\nu_0^{1-2s} \Gamma(2s+1)}{|w|^{d+2s}} \Big| \lVert F\lVert_{L^\infty} - \gamma \Big| 
		\end{align*}
		and the control of $G$ for small velocity follows. For $|w|\geq 1$, we have using \eqref{eq:F0} 
		\begin{align*}
			|G(w)| &\lesssim \int_0^{|w|} e^{-\nu_0 \tau} \nu_0^2 \tau^{4s} \frac{1}{|w|^{d+4s}} \d \tau + ( \lVert F\lVert_{L^{\infty}} -\gamma) \int_{|w|}^{+\infty} e^{-\nu_0 \tau} \nu_0^2 \tau^{2s} \d \tau \\
			&\lesssim \frac{\nu_0^{1-4s}\Gamma(4s+1)}{|w|^{d+4s}} +  e^{-\nu_0 |v|/2} 
		\end{align*}
		which concludes the proof. 
	\end{proof}
	Back to the proof of Proposition \ref{prop:CVop}, we now split the $L^2$-norm in two as follows 
	\begin{align*}
		\int_\Omega \Big( \L^\eps[\psi](x) - \L_{\SR} [\psi](x) \Big)^2 \d x  &=\int_\Omega \bigg( \eps^{-2s} \int_{\RR^d}  \big[ \psi\big( \eta(x,\eps w) \big) - \psi(x) \big] G(w) \d w  \bigg)^2 \d x \\
		&\leq 2\int_\Omega \bigg( \eps^{-2s} \int_{|\eps w|<1}  \big[ \psi\big( \eta(x,\eps w) \big) - \psi(x) \big] G(w) \d w  \bigg)^2 \d x \\
		&\quad + 2\int_\Omega \bigg( \eps^{-2s} \int_{|\eps w|>1}  \big[ \psi\big( \eta(x,\eps w) \big) - \psi(x) \big] G(w) \d w  \bigg)^2 \d x \\\
		&:=2 I_\eps^- + 2 I_\eps^+.
	\end{align*}
	For the integral $I_\eps^+$, the convergence follows from the decay of $G$, namely we have using the Cauchy-Schwarz inequality
	\begin{align*}
		I_\eps^+ &\leq \eps^{-4s} \bigg(\int_{|\eps w|>1} G(w) \d w \bigg) \bigg(\int_{\Omega}\int_{|\eps w|>1}  \big[ \psi\big( \eta(x,\eps w) \big) - \psi(x) \big]^2 G(w) \d w \d x \bigg)
	\end{align*}
	where, using Lemma \ref{lem:F1} we have
	\[ \int_{|\eps w|>1} G(w) \d w  \lesssim \eps^{4s} .\]
	Moreover, with Fubini and Assumption 2 in Definition \ref{def:admissibleOmega} (licit since the domain $\{ |w|>1  \}$ is radially symmetric, hence stable by the change of variable), we get
	\begin{align*}
		\int_{\Omega}\int_{|\eps w|>1 }  \big[ \psi\big( \eta(x,\eps w) \big) - \psi(x) \big]^2 G(w) \d w \d x  &\lesssim \lVert \psi\lVert_{L^2(\Omega)} \int_{|\eps w|>1 } G(w) \d w \\
		&\lesssim \lVert \psi\lVert_{L^2(\Omega)} \eps^{4s}
	\end{align*}
	and the convergence $I_\eps^+ \rightarrow 0$ follows since $\psi\in L^2(\Omega)$. \\
	For the integral $I_\eps^-$, the convergence follows from the regularity of $\psi$ and $\eta$ stated in Assumption 1 of Definition \ref{def:admissibleOmega}. In particular, a direct corollary of the proof of \cite[Lemma A.3]{Cesbron18} shows that if $\psi$ belongs to $\mathfrak{D}^s$ with $s\geq 1/2$ then $D^2_v\big[\psi ( \eta )\big] \in L^2_\mu(\Omega\times V)$ for $V\subset \RR^d$ if $\mu$ is radial and $\mu(V)<+\infty$. Let us focus on the case $s\geq 1/2$ as it is the most difficult one, we will talk about the case $s<1/2$ later on. \\
	A second order Taylor expansion yields 
	\begin{align*}  
		\psi\big( \eta(x,\eps w) \big) - \psi(x)  &= -\eps w\cdot \na [\psi(\eta)](x,0)+ \int_0^1 (1-\tau) D^2_v[\psi(\eta)](x, \eps \tau w) (\eps w,\eps w) \d \tau \\
		&=  -\eps w\cdot \na\psi(x)+ \int_0^1 (1-\tau) D^2_v[\psi(\eta)](x, \eps \tau w) (\eps w,\eps w) \d \tau 
	\end{align*}
	where we notice that since $G$ is radial 
	\[ \int_{|\eps w|<1} w\cdot \na \psi(x) G(w) \d w = 0 .\]
	Hence $I_\eps^-$ can be controlled by 
	\begin{align*}
		I_\eps^- &\leq  \int_\Omega \bigg(\eps^{-2s} \int_{|\eps w|<1} \int_0^1 |\eps w|^2 \big| D^2_v [\psi(\eta)] (x,-\eps \tau w)\big| G(w) \d \tau \d w \bigg)^2 \d y \\
		&\lesssim \eps^{4-4s} \lVert D^2_v[\psi\circ\eta] \lVert^2_{L^2_{|w|^2G(w)}( \Omega\times B)} 
	\end{align*}
	where $B$ is the unit ball centred at $0$. Since $G(w)\lesssim \frac{1}{|w|^{d+2s}}$ for $|w|<1$, we have $\int_{B} |w|^2 G(w) \d w < C<\infty $ hence $I_\eps^-$ converges to $0$.
	
	To conclude this proof, we need to make some remarks about the assumptions on $\psi$. As mentioned in Assumption 1 of Definition \ref{def:admissibleOmega}, we need to assume that $\psi$ satisfies $v\cdot \na_x \psi = 0$ on $\dO$ in order to control the second derivatives in weighted $L^2$ space. However, in the definition of $\mathfrak{D}^s$ we only assume this boundary condition on $\psi$ if $s\geq1/2$. For $s< 1/2$ we do not need this assumption because we can simplify the proof of convergence of $I_\eps^-$ using a first order Taylor expansion instead of a second order due to the fact that $2-4s >0$. In fact, since $\na_v[\psi(\eta)]$ is uniformly bounded, see \cite[Lemma A.3]{Cesbron18}, we can actually conclude the proof for any $\psi\in H^1$. This difference between $s\geq 1/2$ and $s<1/2$ is due to the fact that for $s\geq 1/2$, a function $\psi \in H^s(\Omega)$ has a $L^2$-trace on $\dO$, and it plays a crucial role in the proof of uniqueness of distributional solutions in \cite{Cesbron18}. 
	
\end{proof}

With Proposition \ref{prop:CVop} and Lemma \ref{lem:CVphieps} we can take the limit in the weak formulation \eqref{eq:wfLRSR} and see that the limit $\rho$ satisfies, for all $\psi \in \mathfrak{D}^s$: 
\begin{align} \label{eq:weakLimpbSR}
	\iint_{[0,T)\times\Omega} \rho \Big( \pa_t \psi + \L_{\SR}[\psi] \Big) \d t \d x + \int_{\Omega} \rho_{in} (x)\psi(0,x) \d x = 0 . 
\end{align}
Moreover, we have proved in \cite[Theorem 1.6]{Cesbron18} that such distributional solution is unique, and it is in fact a weak solution in the sense that it satisfies \eqref{eq:weakLimpbSR} for all $\psi\in W^{1,\infty}(0,T;\HSRs(\Omega))$ and it belongs to $L^2 (0,+\infty, \HSRs(\Omega))$. This concludes the proof of Theorem \ref{thm:SR}.

\subsection{The diffuse reflections case}
We consider the rescaled Linear Boltzmann equation \eqref{eq:LinBoltzrescaled} with the diffusive boundary condition \eqref{def:DiffuseBC} in the half-space $\Omega=\RR^d_+$ with equilibrium $F$ satisfying \eqref{eq:F0} with $s>1/2$ in order for the constant $c_0$ in the diffusive boundary condition to be well defined.\\
Given a test function $\phi\in\mathcal{D}(\bO\times\RR^d)$ such that $\gamma_+ \phi = \mathcal{B_{\D}^*}[\gamma_- \phi]$ on $\Gamma_+$, where $\mathcal{B_{\D}^*}$ is given by \eqref{eq:Bstar}, the weak solution $f_\eps$ of \eqref{eq:LinBoltzrescaled}-\eqref{def:DiffuseBC} satisfies, with $Q = (0,+\infty)\times\Omega\times\RR^d $:
\begin{equation} \label{eq:wfLRdiffHS}
	\begin{aligned}
		&\iiint_{Q} f_\eps \pa_t \phi \d t \d x \d v + \iint_{\Omega\times\RR^d} \fin(x,v) \phi (0,x,v) \d x \d v \\
		&= - \eps^{-2s} \iiint_{Q} \Big[ f_\eps \Big( \eps v\cdot \na_x \phi - \nu_0 \phi \Big) + \nu_0 \rho_\eps F(v) \phi \Big] \d t \d x \d v .
	\end{aligned}
\end{equation}
We introduce the operator $A_\eps$ defined as 
\begin{align} \label{eq:AepsDiffHS} 
	A_\eps = \eps v\cdot \na_x - \nu_0 \text{Id} 
\end{align}
on the domain 
\begin{align} \label{eq:DAepsDiff}
	\mathcal{D}(A_\eps) = \lbrace \phi \in L^2_{F}(\Omega\times\RR^d): \, v\cdot \na_x \phi \in  L^2_{F}(\Omega\times\RR^d) \mbox{ and } \gamma_+ \phi = \mathcal{B_{\D}^*} [\gamma_- \phi] \mbox{ on } \Gamma_+ \rbrace .
\end{align}
We recall that the forward exit point and time $x_f$ and $\tau_f$ are defined in \eqref{def:xftf} and we will also write $\tfe (x,v) = \tau_f(x,\eps v)$ and note that $x_f (x,v) = x_f(x,\eps v)$ by definition. We then have the following proposition:
\begin{prop}\label{prop:APDiffHS}
	Given $\psi\in \mathcal{D}\big( [0,+\infty)\times \bO)$, the function $\phi_\eps:= A_\eps^{-1} [-\nu_0\psi]$ can be expressed as
	\begin{equation} \label{def:phiepsDiffHS}
		\begin{aligned}
			\phi_\eps (t,x,v) &= \int_0^{\tfe} \nu_0 e^{-\nu_0\tau} \psi(x+\eps \tau v) \d \tau \\
			&\quad + e^{-\nu_0 \tfe} c_0 \int_{w\cdot n(x_f)<0}\int_0^{+\infty} \nu_0 e^{-\nu_0\tau} \psi(x_f+\eps \tau w) F(w) |w\cdot n(x_f)| \d \tau \d w 
		\end{aligned}
	\end{equation}
	where $\tfe = \tfe(x,v)$ and $x_f=x_f(x,v)$. \\
\end{prop}

\begin{proof}[Proof of Proposition \ref{prop:APDiffHS}]
	The fact that $\phi_\eps$ satisfies the boundary condition of $\mathcal{D}(A_\eps)$ is rather straightforward since for all $(x,v)\in\Gamma_+$ we have $\tfe(x,v) = 0$ hence 
	\begin{align*}
		\gamma_+ \phi_\eps(t,x,v) &= c_0 \int_{w\cdot n(x)<0}\int_0^{+\infty} \nu_0 e^{-\nu_0 \tau} \psi(x+\eps \tau w) F(w) |w\cdot n(x)| \d \tau \d w \\
		&= c_0 \int_{w\cdot n(x)<0} \gamma_- \phi_\eps(x,w)  F(w) |w\cdot n(x)| \d w\\
		&= \BDs [\gamma_- \phi_\eps] 
	\end{align*}
	using the fact that, since $\Omega$ is half space, for any $(x,w)\in\Gamma_-$ we have $\tfe(x,w) = +\infty$ hence 
	\begin{align*} 
		\gamma_- \phi_\eps(x,w) = \int_0^{+\infty} \nu_0 e^{-\nu_0 \tau} \psi(x+\eps \tau w) \d \tau .
	\end{align*}
	Now, let us recall that $v\cdot \na_x \tau_f (x,v) = -1$ and $v\cdot\na_x x_f(x,v) = 0$, proofs of which can be found in \cite{Guo10} in the case of backwards exit time and point (obviously equivalent to the forward ones through the substitution $v\to -v$). Let us now check that $A_\eps \phi_\eps = -\nu_0\psi$ by computing the following: for $(x,v)\in\bO\times\RR^d$:
	\begin{align*}
		\eps v\cdot \na_x \phi_\eps (x,v) &= - \nu_0 e^{-\nu_0 \tfe} \psi(x_f) + \int_0^{\tfe} \nu_0 e^{-\nu_0 \tau} \eps v\cdot \na_x \psi(x+\eps \tau v) \d \tau \\
		&\quad + \nu_0 e^{-\nu_0 \tfe} c_0 \int_{w\cdot n(x_f)<0}\int_0^{+\infty} \nu_0 e^{-\nu_0\tau} \psi(x_f+\eps \tau w) F(w) |w\cdot n(x_f)| \d \tau \d w 
	\end{align*}
	where 
	$$\eps v\cdot \na_x \psi(x+\eps \tau v ) = \frac{\d }{\d \tau} \big[ \psi(x+\eps \tau v) \big]$$
	hence integration by parts yields
	\begin{align*}
		\eps v\cdot \na_x \phi_\eps (x,v) &= - \nu_0 e^{-\nu_0 \tfe} \psi(x_f) + \int_0^{\tfe} \nu_0^2  e^{-\nu_0 \tau} \psi(x+\eps \tau v) \d \tau + \nu_0 e^{-\nu_0 \tfe} \psi(x_f) -\nu_0  \psi(x)  \\
		&\quad + \nu_0 e^{-\nu_0 \tfe} c_0 \int_{w\cdot n(x_f)<0}\int_0^{+\infty} \nu_0 e^{-\nu_0\tau} \psi(x_f+\eps \tau w) F(w) |w\cdot n(x_f)| \d \tau \d w \\
		&= -\nu_0 \psi(x) + \nu_0 \phi_\eps (x,v) 
	\end{align*}
	which concludes the proof. 
\end{proof}
The weak formulation \eqref{eq:wfLRdiffHS} with test function $\phi_\eps=A_\eps^{-1}[-\nu_0\psi]$ expressed in the previous proposition then reads
\begin{equation} \label{eq:wfLRDiff}
	\begin{aligned} 
		\iiint_{Q} f_\eps \pa_t \phi_\eps \d t \d x \d v &+ \iint_{\Omega\times\RR^d} \fin(x,v) \phi_\eps (0,x,v) \d x \d v  \\
		&\quad =  - \iiint_{Q} \eps^{-2s} \rho_\eps \nu_0 (\phi_\eps - \psi) F(v) \d t \d x \d v .
	\end{aligned}
\end{equation}
Using the definition of $\phi_\eps$, we can write the following:
\begin{align*}
	&\int_{\RR^d} \nu_0(\phi_\eps - \psi)  F(v)  \d v \\
	&\quad =  \int_{\RR^d} \bigg[ \int_0^{\tfe} \nu_0 e^{-\nu_0 \tau} \Big( \psi(x+\eps \tau v) - \psi(x) \Big) \d \tau + e^{-\nu_0 \tfe} \Big( \psi(x_f) - \psi(x)\Big) \\
	&\quad +e^{-\nu_0 \tfe} c_0 \int_{w\cdot n<0}\int_0^{+\infty} \nu_0 e^{-\nu_0\tau} \Big( \psi(x_f+\eps \tau w)-\psi (x_f) \Big) F(w) |w\cdot n| \d \tau \d w \bigg] \nu_0 F(v) \d v .
\end{align*}
We decompose this expression in three parts. First, we define the operator $\L_\eps$ as:
\begin{align} 
	\L_\eps [\psi] (x):= & \eps^{-2s}\int_{\RR^d} \int_0^{\tfe} \nu_0^2 e^{-\nu_0 \tau} \Big( \psi(x+\eps \tau v) - \psi(x) \Big) \d \tau F(v) \d v  \nonumber\\
	&= \eps^{-2s}\int_0^{+\infty} \int_{x+\eps \tau v \in \Omega } \nu_0^2 e^{-\nu_0 \tau} \Big( \psi(x+\eps \tau v) - \psi(x) \Big)  F(v) \d v \d \tau  \nonumber\\
	& = \eps^{-2s}\int_0^{+\infty}  \int_{  x+\eps z\in\Omega} \nu_0^2 e^{-\nu_0 \tau} \Big( \psi(x+\eps z) - \psi(x) \Big) \tau^{-d} F\left(\frac{z}{\tau}\right) \d v \d \tau \nonumber \\
	& = \eps^{-2s} \int_{ x+\eps z\in\Omega} \Big( \psi(x+\eps z) - \psi(x) \Big) F_1(z) \d z \label{def:LepsDiffHS} 
\end{align}
with $F_1$ defined in \eqref{eq:F1}. Second, we define $\kappa_\eps$ as
\begin{align}
	\kappa_\eps[\psi](x) = \int_{\RR^d} \nu_0 e^{-\nu_0 \tfe} \Big(  \psi(x) - \psi(x_f) \Big) F(v) \d v . \label{def:kappaepsDiffHS}
\end{align}
And finally, we notice that
\begin{align*}
	& \int_{w\cdot n<0}\int_0^{+\infty} \nu_0 e^{-\nu_0\tau} \Big( \psi(x_f+\eps \tau w)-\psi (x_f) \Big)  F(w)  |w\cdot n| \d \tau \d w  \\
	&= \int_0^{+\infty} \int_{z\cdot n<0} \Big( \psi(x_f+\eps z)-\psi (x_f) \Big)  \nu_0 e^{-\nu_0\tau} \tau^{-d-1}  F\left(\frac{z}{\tau}\right) |z\cdot n|  \d z  \d \tau  \\
	&= \int_{z\cdot n<0} \Big( \psi(x_f+\eps z)-\psi (x_f) \Big) F_0(z) |z\cdot n|  \d z \label{def:DepsDiffHS}
\end{align*}
where
\begin{equation} \label{def:F0DiffHS}
	F_0 (z) = \int_0^{+\infty} \nu_0 e^{-\nu_0\tau} \tau^{-d-1} F\left(\frac{z}{\tau}\right)\d \tau.
\end{equation}
We introduce the substitution $\mathcal{P}_\eps$ defined as
\begin{equation} \label{eq:CdVPeps}
	\mathcal{P}_\eps : (x,v) \in\Omega\times\RR^d \mapsto (\tau,y,w)= \big(\tfe(x,v),x_f(x,v),-v\big)\in \RR_+\times \Gamma_-
\end{equation}
for which we have, using the expressions of $\na x_f$ and $\na \tau_f$ deduced from \cite{Guo10}, $|\det \na \mathcal{P}_\eps^{-1} |= \eps |v\cdot n(y)|$ and $x$ becomes $x=y+\eps \tau w$. We get
\begin{align*}
	&\int_\Omega \rho_\eps(x) \eps^{-2s} c_0 \int_{\RR^d} \nu_0e^{-\nu_0 \tfe} F(v)  \int_{z\cdot n<0} \Big( \psi(x_f+\eps z)-\psi (x_f) \Big) F_0(z) |z\cdot n|  \d z\d v \d x \\
	&=  \int_{\dO} \bigg( c_0 \int_{w\cdot n<0} \int_0^{+\infty} \rho_\eps(y + \eps \tau w)\nu_0 e^{-\nu_0 \tau } F(w) |w\cdot n| \d \tau \d w \bigg) \\
	&\quad \times \bigg( \eps^{1-2s}  \int_{z\cdot n<0} \Big( \psi(y+\eps z)-\psi (y) \Big) F_0(z) |z\cdot n(y)|  \d z \bigg) \d \sigma(y) \\
	&= - \int_{\dO} A_\eps^{-1}[-\nu_0\rho_\eps](y,\cdot) \Deps [\psi](y) \cdot n(y) \d \sigma(y) 
\end{align*}
recognising
\begin{align*}
	c_0 \int_{w\cdot n<0} \int_0^{+\infty} \rho_\eps(y + \eps \tau w)\nu_0 e^{-\nu_0 \tau } F(w) |w\cdot n| \d \tau \d w &= \mathcal{B_{\D}^*} \left[ A_\eps^{-1} [-\nu_0 \rho_\eps] \right] \\
	&= A_\eps^{-1} [-\nu_0 \rho_\eps]
\end{align*}
on $\Gamma_-$ by construction. The operator $\Deps$ was introduced in \cite{CesbronMelletPuel} as 
\[ \Deps [\psi](y) = \eps^{1-2s} \int_{z\cdot n<0} \Big( \psi(y+\eps z)-\psi (y) \Big) F_0(z) z  \d z .\]
Note that $A_\eps^{-1} [-\nu_0\rho_\eps](y,\cdot ) = A_\eps^{-1} [-\nu_0\rho_\eps](y,v ) $ for all $(y,v) \in \Gamma_-$, independent of $v$ by definition of the operator $\BDs$.\\
Altogether, the weak formulation of \eqref{eq:LinBoltzrescaled}-\eqref{def:DiffuseBC} becomes
\begin{equation} \label{eq:wfLRdiffHS2}
	\begin{aligned}
		&\iiint_{Q} f_\eps \pa_t \phi_\eps\d t \d x \d v -\iint_{\Omega\times\RR^d} f_{in} \phi_\eps(0,x,v) \d x \d v \\
		&\quad =-  \iint_{\RR_+\times\Omega}\rho_\eps \Big( \L_\eps[\psi](x) - \kappa_\eps[\psi(x)] \Big)  \d t \d x \\
		&\qquad +\int_{\RR_+\times\dO} A_\eps^{-1}[-\nu_0\rho_\eps](y,\cdot) \Deps [\psi](y) \cdot n(y) \d \sigma(y) \d t .
	\end{aligned}
\end{equation}

\begin{remark} \label{rmk:Leps}
	We notice that $\nu_0 e^{-\nu_0 \tfe} = \int_{\tfe}^{+\infty} \nu_0^2e^{-\nu_0 \tau} \d \tau$ hence we see that we recover here the operator $\L^\eps$ defined in \cite{CesbronMelletPuel}. Indeed, if we defined an extension $\widetilde{\psi}(x+\eps v,v) = \psi(x+\eps v)$ if $x+\eps v\in\Omega $ and $\widetilde{\psi}(x+\eps v,v) = \psi(x_f(x,v))$ is $x+\eps v \notin \Omega$ then we have 
	\begin{align*}
		\L_\eps[\psi](x) - \kappa_\eps[\psi](x) = \L^\eps [\psi](x) = \eps^{-2s} \int_{\RR^d} \Big( \widetilde{\psi} (x+\eps v,v) - \psi(x) \Big) F_1(v) \d v. 
	\end{align*}
\end{remark}

\subsubsection{Macroscopic limit, proof of Theorem \ref{thm:Diff}} 

We now wish to take the limit in \eqref{eq:wfLRdiffHS2} as $\eps$ goes to $0$. We will prove the convergence of the terms on the left-hand-side in a more general setting in Appendix \ref{subsection:CVPhieps} so we focus now on the terms on the right-hand-side. 

\begin{prop} \label{prop:CvLepsdiff}
	For all $\psi \in L^\infty(0,T; H^2(\Omega))$ such that $\L \psi \in L^2 (\RR_+\times\Omega)$ we have 
	\begin{align}
		&\L_\eps [\psi](t,x) \rightarrow -\gamma_{d,s}\Delsreg [\psi](t,x) &\quad \mbox{in } L^2((0,T)\times\Omega)\mbox{-strong for all } T>0.  \label{eq:CvLepsdiff} \\
		& \kappa_\eps[\psi](t,x) \rightarrow \kappa[\psi](t,x) &\quad \mbox{in } L^2((0,T)\times\Omega)\mbox{-strong for all } T>0.  \label{eq:CVkappadiff}
	\end{align}
\end{prop}
\begin{proof}
	As mentioned in Remark \ref{rmk:Leps} above, the complete operator $\L_\eps+\kappa_\eps$ was already studied in \cite{CesbronMelletPuel}. In particular, it was proved in \cite[Proposition 2.2]{CesbronMelletPuel} that for all  $T>0$
	\begin{align*}
		\L_\eps [\psi](t,x) - \kappa_\eps[\psi](t,x) \rightarrow -\gamma_{d,s}\Delsreg [\psi](t,x) - \kappa[\psi](t,x) &\quad \mbox{in } L^2((0,T)\times\Omega)\mbox{-strong}
	\end{align*}
	so we will only show \eqref{eq:CvLepsdiff} to prove Proposition \ref{prop:CvLepsdiff}.   \\
	Let us first notice that we can write $\Delsreg[\psi]$ as
	\begin{align*}
		-\gamma_{d,s}\Delsreg [\psi](x) &= \gamma_1 \pv \int_{x+v\in\Omega} \frac{ \psi(x+v) - \psi(x) }{|v|^{d+2s}} \d v \\
		&= \eps^{-2s} \pv \int_{x+\eps v\in\Omega} \Big(\psi(x+\eps v) - \psi(x) \Big) \frac{\gamma_1}{|v|^{d+2s}} \d v .
	\end{align*}
	We can then write (omitting the $t$ variable for the sake of clarity)
	\begin{align*}
		&\int_\Omega \Big( \L_\eps[\psi] (x) + \gamma_{d,s} \Delsreg [\psi](x) \Big)^2 \d x \\
		&\quad = \int_\Omega \bigg(\eps^{-2s} \int_{x+\eps v\in\Omega}  \Big(\psi(x+\eps v) - \psi(x) \Big) G(v) \d v \bigg)^2 \d x \\
		&\quad \leq 2 \int_\Omega \bigg(\eps^{-2s} \int_{x+\eps v\in\Omega,\, |\eps^\beta v|<1}  \Big(\psi(x+\eps v) - \psi(x) \Big) G(v) \d v \bigg)^2 \d x \\
		&\qquad + 2 \int_\Omega \bigg(\eps^{-2s} \int_{x+\eps v\in\Omega,\, |\eps^\beta v|>1}  \Big(\psi(x+\eps v) - \psi(x) \Big) G(v) \d v \bigg)^2 \d x\\
		&\quad := 2 I_\eps^- + 2 I_\eps^+ 
	\end{align*}
	for some $\beta\in(0,1)$ to be chosen later on, and with $G(v)=F_1(v)- \frac{\gamma_1}{|v|^{d+2s}}$ whose behaviour is given by Lemma \ref{lem:F1}.  The convergence of $I_\eps^+$ follows from the decay of $G$: 
	\begin{align*}
		I_\eps^+ \leq \eps^{-4s}\bigg( \iint_{x+\eps v\in\Omega,\, |\eps^\beta v|>1}  \Big(\psi(x+\eps v) - \psi(x) \Big)^2 G(v) \d v \d x \bigg) \bigg(\int_{|\eps^\beta v|>1} G(v) \d v \bigg) 
	\end{align*}
	where on the one hand
	\begin{align*}
		\int_{|\eps^\beta v|>1} G(v) \d v \lesssim \eps^{4s\beta} 
	\end{align*}
	and on the other, with the substitution $x\to y=x+\eps v \in \Omega$ 
	\begin{align*}
		\iint_{x+\eps v\in\Omega,\, |\eps^\beta v|>1}  \Big(\psi(x+\eps v) - \psi(x) \Big)^2 G(v) \d v \d x &\lesssim \lVert \psi\lVert_{L^2(\Omega)}^2 \int_{|\eps^\beta v|>1} G(v) \d v \\
		&\lesssim \eps^{4s\beta} \lVert \psi\lVert_{L^2(\Omega)}^2
	\end{align*}
	hence $I_\eps^+ \leq C \eps^{4s\beta} \lVert \psi\lVert_{L^2(\Omega)}$. The convergence of $I_\eps^{-}$ follows from the regularity of $\psi$. A second order Taylor expansion reads
	\begin{align*}
		\psi(x+\eps v) - \psi(x) = \eps v\cdot \na_x \psi(x) + \int_0^1 (1-\tau) D^2_x\psi(x+\eps \tau v)(\eps v,\eps v) \d \tau .
	\end{align*} 
	For the first order term, we notice that since $\Omega$ is a half-space, for any fixed $x$ the set $\lbrace v:\, x+\eps v\in\Omega,\, |\eps^\alpha v|<1\rbrace$ is invariant by the substitution $v=(v',v_d) \to \tilde{v}=(-v',v_d)$ hence, by radial symmetry of $G$, 
	\begin{align*}
		\pv \int_{x+\eps v\in\Omega,\, |\eps^\beta v|<1} v_i \na_x \psi(x) G(v) \d v = 0 \quad \mbox{ for all } i=1,\dots, d-1.
	\end{align*}
	As a result, we have 
	\begin{align*}
		\pv \int_{x+\eps v\in\Omega,\, |\eps^\beta v|<1} \eps v\cdot \na_x\psi(x) G(v) \d v  = \pv \int_{x+\eps v\in\Omega,\, |\eps^\beta v|<1} \eps v_d \pa_n \psi(x) G(v) \d v 
	\end{align*}
	where $\pa_n$ is the normal derivative at the boundary. Moreover, if $\eps v_d<x_d$ then both $x+\eps v\in\Omega$ and $x-\eps v\in\Omega$ so again the integral cancels thanks to the symmetry of $G$, we are left with 
	\begin{align*}
		&\pv \int_{x+\eps v\in\Omega,\, |\eps^\beta v|<1} \eps v\cdot \na_x\psi(x) G(v) \d v \lesssim \mathds{1}_{\{ x_d<\eps^{1-\beta}\}} \pa_n \psi(x) \eps \int_{ \eps^{-1}x_d < v_d <\eps^{-\beta}} v_d G(v) \d v  \\
		&\lesssim \mathds{1}_{\{ x_d<\eps^{1-\beta}\}} \pa_n \psi(x) \eps \bigg( \int_{\eps^{-1}x_d<v_d<1} v_d G(v) \d v + \int_{1<v_d<\eps^{-\beta}} v_d G(v) \d v  \bigg)\\
		&\lesssim \mathds{1}_{\{ x_d<\eps^{1-\beta}\}} \pa_n \psi(x) \eps \big( \eps^{2s-1} x_d^{1-2s} + \eps^{4s \beta} \big) \\
		&\lesssim \eps^{2s}  \mathds{1}_{\{ x_d<\eps^{1-\beta}\}} x_d^{1-2s} \pa_n \psi(x)
	\end{align*}
	for any $\beta>1/2$, where we used Lemma \ref{lem:F1} to estimates the integrals. This yields, for $I_\eps^-$: 
	\begin{align*}
		I_\eps^{-} &\lesssim \int_\Omega \bigg( \mathds{1}_{\{ x_d<\eps^{1-\beta}\}}  x_d^{1-2s} \pa_n \psi(x) \\
		&\quad + \eps^{-2s} \int_{x+\eps v\in\Omega, \, |\eps^\beta v|<1 } |\eps v|^2 \int_0^1 \big| D^2_x\psi(x+\eps \tau v)\big|^2 G(v) \d v \bigg)^2 \d x\\
		&\lesssim \eps^{1-\beta}  \int_\Omega \big| x_d^{1-2s}  \pa_n \psi \big|^2 \d x + \eps^{4-4s} \lVert D^2_x\psi \lVert^2_{L^2(\Omega)} \bigg( \int_{|\eps^{\beta}v|<1} |v|^2G(v) \d v \bigg)^2
	\end{align*}
	which goes to $0$ when $\eps$ goes to $0$ for all $\beta<1$ because it was proved in \cite[Proposition 3.3]{CesbronMelletPuel} that the assumption $\L [\psi]\in L^2(\Omega)$ implies 
	\[ \int_\Omega  \big| x_d^{1-2s}  \pa_n \psi \big|^2 \d x  < +\infty .\]
	Choosing $\beta \in (1/2,1)$ concludes the proof. 
\end{proof}
\begin{remark}
	Note that this proof of \eqref{eq:CvLepsdiff} is actually rather similar to the one of \cite[Proposition 2.2]{CesbronMelletPuel}, the main difference is that we do not use here the extension $\widetilde{\psi}$ defined in Remark \ref{rmk:Leps}. One could also prove \eqref{eq:CVkappadiff} directly without using the extension, the arguments of this proof would be analogous to the proof of \eqref{eq:CvLepsdiff} which is why we chose not to write it here as to avoid unnecessary repetitions. 
\end{remark}

\begin{prop} \label{prop:CvDepsDiffHS}
	For all $\psi$ such that $\L [\psi] \in L^2(\Omega)$ and $\Dlim [\psi](y)\cdot n(y) = 0$ on $\dO$
	\begin{equation}
		\lim_{\eps\to 0} \eps^{-1} \int_{\dO} \Big| \Deps [\psi](y)\cdot n(y) \Big|^2 \d \sigma(y)  = 0 
	\end{equation}
\end{prop}
\begin{proof}
	We recall that $\Dlim$ is defined in \eqref{def:Dlim}. We introduce the notation $G_0 (w) =  | F_0(w) - \frac{\gamma_0}{|w|^{d+2s}}|$ which satisfies: 
	\begin{align} \label{eq:G0} 
		\forall w\in\RR^d :\quad  G_0(w) \lesssim \frac{1}{|w|^{d+2s}}  \quad \mbox{ and } \quad \forall |w|>1 :\quad G_0(w) \lesssim \frac{1}{|w|^{d+4s}} .
	\end{align}
	Note that these estimates can be proved via is a simpler version of the proof of Lemma \ref{lem:F1} so we will not write the proof explicitly. Since $\Dlim [\psi](y)\cdot n = 0$ on $\dO$ we can write 
	\begin{align*}
		&\eps^{-1} \int_{\dO} \Big| \Deps [\psi](y)\cdot n(y) \Big|^2 \d \sigma(y) \\
		& = \eps^{-1} \int_{\dO} \bigg( \eps^{1-2s} c_0 \int_{w\cdot n<0} \Big(\psi(y+\eps w)-\psi(y) \Big) \Big( F_0(w) - \frac{\gamma_1}{|w|^{d+2s}} \Big) |w\cdot n| \d w  \bigg)^2 \d \sigma(y)\\
		&\leq 2\eps^{-1} \int_{\dO} \bigg( \eps^{1-2s} c_0 \int_{w\cdot n<0,\, |\eps w|<1} \Big(\psi(y+\eps w)-\psi(y) \Big) G_0(w) |w\cdot n| \d w  \bigg)^2 \d \sigma(y) \\
		&\quad + 2\eps^{-1} \int_{\dO} \bigg( \eps^{1-2s} c_0 \int_{w\cdot n<0,\, |\eps w|>1} \Big(\psi(y+\eps w)-\psi(y) \Big) G_0(w) |w\cdot n| \d w  \bigg)^2 \d \sigma(y)\\
		&:= I_\eps^- + I_\eps^+
	\end{align*}
	For $I_\eps^+$, we use the decay of $G_0$ with the substitution $w\to z= \eps w $ to write
	\begin{align*}
		I_\eps^+ &\lesssim \eps^{-1} \int_{\dO}\bigg( \int_{z\cdot n<0,\, |z|>1}  \Big(\psi(y+z)-\psi(y) \Big) \eps^{-d-2s} G_0\left(\frac{z}{\eps}\right) |z\cdot n| \d z \bigg)^2 \d \sigma (y)\\
		&\lesssim \eps^{-1}\int_{\dO}\bigg( \int_{z\cdot n<0,\, |z|>1}  \Big(\psi(y+z)-\psi(y) \Big) \eps^{2s} \frac{|z\cdot n|}{|z|^{d+4s}}  \d z \bigg)^2 \d \sigma (y)\\
		&\lesssim \eps^{4s-1} \int_{\dO} \bigg( \int_{z\cdot n<0,\, |z|>1}  \Big(\psi(y+z)-\psi(y) \Big)^2 \frac{|z\cdot n|}{|z|^{d+4s}}  \d z \bigg) \bigg(\int_{|z|>1}\frac{|z\cdot n|}{|z|^{d+4s}} \d z \bigg) \d \sigma (y)\\
		&\lesssim \eps^{4s-1}\int_{\dO} \int_{z\cdot n<0,\, |z|>1} \int_0^1 \big| z \cdot\na \psi(y+\tau z) \big|^2 \d \tau \frac{|z\cdot n|}{|z|^{d+4s}}  \d z  \d \sigma (y)
	\end{align*}
	With the substitution $\mathcal{P}^{-1}_1: (y,\tau,z) \in \dO\times[0,1]\times\RR^d \to (y+\tau z,z)\in\Omega\times\RR^d$, this yields 
	\begin{align*}
		I_\eps^+ &\lesssim \eps^{4s-1} \int_\Omega \int_{z\cdot n<0,\, |z|>1}\big| \na \psi(x)\big|^2 \frac{\gamma_1 |z|^2 }{|z|^{d+4s}} \d z  \d x \\
		&\lesssim \eps^{4s-1} \lVert \na \psi \lVert_{L^2(\Omega)} 
	\end{align*} 
	For $I_\eps^-$ we want to use the regularity of $\psi$ through a second order Taylor expansion. However, since we integrate on $\Gamma_-$, we need to isolate the normal derivative so we write $w=(w',w_d)\in\RR^{d-1}\times\RR_+$. Furthermore, with the notation $\widetilde{w}=(-w',w_d)$ we have $G_0(\widetilde{w})|\widetilde{w}\cdot n| = G_0(w)|w\cdot n|$ and the domain of integration is invariant by the substitution $w\to \widetilde{w}$ hence 
	\begin{align*}
		&\int_{w\cdot n<0,\, |\eps w|<1} \Big(\psi(y+\eps w)-\psi(y) \big) G_0(w)|w\cdot n|\d w \\
		&= \frac{1}{2} \int_{w\cdot n<0, \, |\eps w|<1} \Big( \psi(y+\eps w) + \psi(y+\eps \widetilde{w}) - 2\psi(y) \Big) G_0(w) |w\cdot n| \d w \\
		&= \frac{1}{2} \int_{w\cdot n<0, \, |\eps w|<1} \bigg( \eps w\cdot \na \psi(y) +\int_0^1 (1-\tau) D^2\psi(y+\eps \tau w)(\eps w, \eps w) \d \tau  \\
		&\qquad + \eps \widetilde{w}\cdot \na \psi(y) +\int_0^1 (1-\tau) D^2\psi(y+\eps \tau \widetilde{w})(\eps \widetilde{w}, \eps \widetilde{w}) \d \tau \bigg) G_0(w) |w\cdot n| \d w \\
		& = \int_{w\cdot n<0, \, |\eps w|<1} \bigg( \eps (w\cdot n)\pa_n \psi(y) + \int_0^1  D^2\psi(y+\eps \tau w)(\eps w, \eps w) \d \tau \bigg) G_0(w) |w\cdot n| \d w .   
	\end{align*}
	Using the Cauchy-Schwarz inequality, this yields
	\begin{align*}
		&I_\eps^- \leq 2\eps^{3-4s} \int_{\dO} \int_{w\cdot n<0, \, |\eps w|<1} \big| \pa_n \psi(y) \big|^2 |w\cdot n|^2 G_0(w) \d w \d \sigma(y) \int_{|\eps w|<1} |w|^2 G_0(w) \d w  \\
		&+ 2\eps^{5-4s} \int_{\dO}  \int_{w\cdot n<0, \, |\eps w|<1}  \int_0^1  \big|D^2\psi(y+\eps \tau w)\big|^2 |w|^2 G_0(w) |w\cdot n|  \d \tau \d w \int_{|\eps w|<1} |w|^3 G_0(w) \d w .   
	\end{align*}
	where, using the substitution $\mathcal{P}_\eps^{-1}: (y,\tau,z) \in \dO\times[0,1]\times\RR^d \to (y+\eps \tau z,z)\in\Omega\times\RR^d$, we have 
	\begin{align*}
		&\int_{\dO}  \int_{w\cdot n<0, \, |\eps w|<1}  \int_0^1  \big|D^2\psi(y+\eps \tau w)\big|^2 |w|^2 G_0(w) |w\cdot n|  \d \tau \d w \\
		&\quad  \leq \eps^{-1} \int_\Omega \int_{w\cdot n<0, \, |\eps w|<1} \big|D^2 \psi(x)\big|^2 |w|^2 G_0(w) \d w \d x \\
		&\quad \leq \eps^{-1} \lVert D^2 \psi \lVert_{L^2(\Omega)} \int_{|\eps w|<1} |w|^2 G_0(w) \d w. 
	\end{align*}
	Moreover, for $\alpha>2s$, the partial $\alpha$-moment $M_\alpha^\eps$ of $G_0$ can be bounded as
	\begin{align*}
		M_\alpha^\eps(G_0) &:= \int_{|\eps w|<1} |w|^\alpha G_0(w) \d w \\
		&\lesssim \int_{|w|<1} \frac{|w|^\alpha}{|w|^{d+2s}} \d w + \int_{1<|w|<\eps^{-1}} \frac{|w|^\alpha}{|w|^{d+4s}}\d w \lesssim 1 + \eps^{4s-\alpha} 
	\end{align*}
	hence for $\alpha=3$ and $1/2<s<3/4$, $M_3^\eps(G_0) = C \eps^{4s-3}$, and otherwise $M_\alpha^\eps(G_0)\leq C <+\infty$. This yields  
	\begin{align*}
		I_\eps^- &\lesssim\eps^{3-4s} \lVert \pa_n \psi \lVert_{L^2(\dO)} + \eps^{4-4s}M_3^\eps(G_0) \lVert D^2 \psi \lVert_{L^2(\Omega)} .
	\end{align*}
	Finally, if $1/2<s<3/4$ we get 
	\begin{align*}
		I_\eps^- \lesssim \eps^{3-4s} \lVert \pa_n \psi \lVert_{L^2(\dO)} +  \eps \lVert D^2\psi \lVert_{L^2(\Omega)} 
	\end{align*}
	and for $s>3/4$ we have 
	\begin{align*}
		I_\eps^- \lesssim \eps^{3-4s}  \lVert \pa_n \psi \lVert_{L^2(\dO)} +  \eps^{4-4s} \lVert D^2\psi \lVert_{L^2(\Omega)} .
	\end{align*}
	For $1/2<s<3/4$, this concludes the proof since $3-4s>0$, $\lVert \pa_n \psi\lVert_{L^2(\dO)} <C$ and $\lVert D^2 \psi\lVert <C$ because we assume $\psi \in H^2(\Omega)$. Furthermore, for $s>3/4$ we see that $3-4s<0$ but in that case we proved in \cite[Proposition 3.3]{CesbronMelletPuel} that $\pa_n \psi(x) = 0$ on $\dO$ when $s>3/4$ so that $\lVert \pa_n \psi \lVert_{L^2(\dO)} = 0$ and the convergence follows. 
\end{proof}

Finally, we conclude this section with the proof of Theorem \ref{thm:Diff}:
\begin{proof}[Proof of Theorem \ref{thm:Diff}]
	The only thing left to prove is that 
	\[ \int_0^T \int_{\dO} A_\eps^{-1}[-\nu_0\rho_\eps] (y,\cdot) \Deps [\psi](y)\cdot n(y) \d \sigma(y) \d t \longrightarrow 0 .\]
	This convergence follows from Proposition \ref{prop:CvDepsDiffHS} and the fact that 
	\begin{align*}
		&\int_{\dO} \Big| A_\eps^{-1} [\nu_0\rho_\eps] (y,\cdot) \Big|^2 \d \sigma(y) \\
		&\quad= \int_{\dO} \bigg( c_0 \int_{w\cdot n<0} \int_0^{+\infty} \nu_0 e^{-\nu_0 \tau} \rho_\eps(y+\eps \tau w) F(w) |w\cdot n|\d w \d \tau \bigg)^2 \d \sigma(y) \\
		&\quad\lesssim\int_{\dO} \int_{w\cdot n<0} \int_0^{+\infty} \nu_0^2 e^{-2\nu_0 \tau} \big| \rho_\eps(y+\eps \tau w)\big|^2 F(w)|w\cdot n| \d \tau \d w \d \sigma(y) \\
		&\quad\lesssim \frac{1}{\eps} \int_\Omega \int_{\RR^d} |\rho_\eps(x)|^2 \nu_0 ^2e^{-2\nu_0 \tfe} F(w) \d w \d x \\
		&\quad\lesssim \frac{1}{\eps} \lVert \rho_\eps \lVert_{L^2(\Omega)} .
	\end{align*}
	
\end{proof}

\section{The Maxwell boundary condition}

We consider the rescaled Linear Boltzmann equation \eqref{eq:LinBoltzrescaled} with the Maxwell boundary condition \eqref{def:MaxwellBC} for some $\alpha\in(0,1)$ in the half-space $\Omega=\RR^d_+$ with equilibrium $F$ satisfying \eqref{eq:F0} and $s>1/2$ in order for the constant $c_0$ in the diffusive boundary condition to be well defined.\\
Given a test function $\phi\in\mathcal{D}(\bO\times\RR^d)$ such that $\gamma_+ \phi = \mathcal{B}^*_\alpha[\gamma_- \phi]$ on $\Gamma_+$, where $\mathcal{B}^*_\alpha$ is given by \eqref{def:Bdiffstar}, the weak solution $f_\eps$ of \eqref{eq:LinBoltzrescaled}-\eqref{def:MaxwellBC} satisfies, with $Q = (0,+\infty)\times\Omega\times\RR^d $:
\begin{equation} \label{eq:wfLRMaxHS}
	\begin{aligned}
		&\iiint_{Q} f_\eps \pa_t \phi \d t \d x \d v + \iint_{\Omega\times\RR^d} \fin(x,v) \phi (0,x,v) \d x \d v \\
		&= - \eps^{-2s} \iiint_{Q} \Big[ f_\eps \Big( \eps v\cdot \na_x \phi - \nu_0 \phi \Big) + \nu_0 \rho_\eps F(v) \phi \Big] \d t \d x \d v .
	\end{aligned}
\end{equation}
We introduce the operator $A_\eps$ defined as 
\begin{align} \label{eq:AepsMaxHS} 
	A_\eps = \eps v\cdot \na_x - \nu_0 \text{Id}
\end{align}
on the domain 
\begin{align} \label{eq:DAepsMax}
	\mathcal{D}(A_\eps) = \lbrace \phi \in L^2_{F}(\Omega\times\RR^d): \, v\cdot \na_x \phi \in  L^2_{F}(\Omega\times\RR^d) \mbox{ and } \gamma_+ \phi = \mathcal{B}^*_\alpha [\gamma_- \phi] \mbox{ on } \Gamma_+ \rbrace .
\end{align}

\begin{prop}\label{prop:APDiffConv}
	Given $\psi\in\mathcal{D}([0,T)\times\bO)$ the function $\phi_\eps:= A_\eps^{-1}[-\nu_0 \psi]$ can be expressed as 
	\begin{equation} \label{def:phiepsDiffConv}
		\begin{aligned}
			\phi_\eps (t,x,v) &= \int_0^{\tfe} \nu_0 e^{-\nu_0\tau} \psi(x+\eps\tau v) \d \tau \\
			&\quad + (1-\alpha) e^{-\nu_0 \tfe} A_\eps^{-1}[-\nu_0 \psi] (x_f, \mathcal{R}_{x_f} v) \\
			&\quad + \alpha e^{-\nu_0 \tfe} c_0 \int_{w\cdot n(x_f)<0} A_\eps^{-1}[-\nu_0\psi](x_f,w) F(w) |w\cdot n(x_f)|  \d w 
		\end{aligned}
	\end{equation}
	with $\tfe = \tfe(x,v)$ and $x_f=x_f(x,v)$ and where, for any $(y,w)\in\Gamma_-$:
	\[ A_\eps^{-1}[-\nu_0\psi](y,w) = \int_0^{+\infty} \nu_0 e^{-\nu_0 \tau} \psi(y + \eps \tau w) \d \tau.\]
\end{prop}
\begin{proof}
	The proof of this Proposition is a direct corollary of the proofs of Proposition \ref{prop:solAP} and Proposition \ref{prop:APDiffHS}.
\end{proof}
We then write the operators in the weak formulation as 
\begin{align*}
	&\int_{\RR^d} \nu_0 ( \phi_\eps - \psi ) F(v) \d v \\
	&= \int_{\RR^d} \nu_0 \bigg( \int_0^{\tfe} \nu_0 e^{-\nu_0 \tau} \Big( \psi(x+\eps \tau v) -\psi(x)\Big) \d \tau +  \alpha  e^{-\nu_0 \tfe} \Big( \psi(x_f) -\psi(x) \Big)  \bigg)F(v) \d v    \\
	&\quad + (1-\alpha) \int_{\RR^d} \nu_0  e^{-\nu_0 \tfe} \Big(A_\eps^{-1}[\nu_0 \psi] (x_f,\mathcal{R}_{x_f}v) - \psi(x) \Big)  F(v) \d v   \\
	&\quad + \alpha \int_{\RR^d} \nu_0 e^{-\nu_0\tfe} c_0 \int_{w\cdot n(x_f)<0} \Big( A_\eps^{-1}[\nu_0\psi] (x_f,w) - \psi(x_f) \Big) |w\cdot n(x_f)| F(w) \d w F(v)  \d v 
\end{align*}
Let us regroup these terms in the following way. First of all, we define as usual $\L^\eps_{\SR} $ as 
\begin{align*}
	(1-\alpha) \L^\eps_{\SR} [\psi](x) &:= \eps^{-2s} (1-\alpha) \int_{\RR^d} \bigg( \int_0^{\tfe} \nu_0^2 e^{-\nu_0 \tau} \Big( \psi(x+\eps \tau v) -\psi(x)\Big) \d \tau \\
	&\quad +  \nu_0 e^{-\nu_0 \tfe} \Big(A_\eps^{-1}[\nu_0 \psi] (x_f,\mathcal{R}_{x_f}v) - \psi(x) \Big)  \bigg)F(v)  \d v  \\
	&= \eps^{-2s} (1-\alpha) \int_{\RR^d} \int_0^{+\infty} \nu_0^2 e^{-\nu_0 \tau} \Big(\psi\big(\eta(x,\eps \tau v)\big) - \psi(x) \Big) F(v) \d \tau \d v \\
	&= \eps^{2s} (1-\alpha) \int_{\RR^d}  \Big(\psi\big(\eta(x,\eps v)\big) - \psi(x) \Big) F_1(v) \d v
\end{align*}
with $F_1$ defined in \eqref{eq:F1}. Next, we define $\keps$ as 
\begin{align*}
	\alpha \keps[\psi](x) = \eps^{-2s}\alpha \int_{\RR^d}  \nu_0 e^{-\nu_0 \tfe} \Big( \psi(x)- \psi(x_f)  \Big) F(v) \d v 
\end{align*}
and $\L^\eps$ as 
\begin{align*}
	\alpha \L^\eps[\psi](x) &= \eps^{-2s} \alpha \int_{\RR^d} \int_0^{\tfe} \nu_0^2 e^{-\nu_0 \tau} \Big( \psi(x+\eps \tau v) -\psi(x)\Big) F(v) \d \tau \d v \\
	&= \eps^{-2s} \alpha \int_{x+\eps v \in \Omega} \Big( \psi(x+\eps v) - \psi(x) \Big) F_1(v) \d v .
\end{align*}
Recognising the boundary term of the diffusive case, the weak formulation of \eqref{eq:LinBoltzrescaled}-\eqref{def:MaxwellBC} then reads
\begin{equation} \label{eq:wfLRMaxwellHS}
	\begin{aligned}
		&\iiint_{Q} f_\eps \pa_t \phi_\eps\d t \d x \d v + \iint_{\Omega\times\RR^d} f_{in} \phi_\eps(0,x,v) \d x \d v \\
		&=-  \iint_{(0,T)\times\Omega}\rho_\eps \Big( (1-\alpha) \L^\eps_{\SR} [\psi] (x) + \alpha( \L^\eps[\psi](x) - \kappa_\eps[\psi(x)] )\Big)  \d t \d x \\
		&\quad -\alpha \int_{(0,T)\times\dO} A_\eps^{-1}[\nu_0\rho_\eps](y,\cdot) \Deps [\psi](y) \cdot n(y) \d \sigma(y) \d t .
	\end{aligned}
\end{equation}
The macroscopic limit of this weak formulation then follows from the convergence of test functions proved in a general setting in Appendix \ref{subsection:CVPhieps} and the convergences of the operators proved in the previous sections. More precisely, Proposition \eqref{prop:CVop} proves that $\L_{\SR}^\eps[\psi]$ converges strongly in $L^2((0,+\infty)\times\Omega)$ to $\L_{\SR}[\psi]$ defined in \eqref{def:LSR}, Proposition \ref{prop:CvLepsdiff} proves that $\L^\eps [\psi]- \kappa_\eps[\psi]$ converges in the same sense to $\L_{\D}[\psi]$ defined in \eqref{def:LDlim} and Proposition \ref{prop:CvDepsDiffHS} establishes the convergence of the boundary term. Altogether, we get in the limit that $\rho$ satisfies 
\begin{equation} \label{eq:weakLimpbMaxwell} 
	\iint_{(0,T)\times\Omega} \rho\Big(  \pa_t \psi + (1-\alpha) \L_{\SR}[\psi] + \alpha \L_{\D}[\psi] \Big) \d t \d x 
	- \int_\Omega \rho_{in} \psi(0,x) \d x = 0 
\end{equation}
for all $\psi\in W^{1,\infty}(0,+\infty ; H^2(\Omega))$ such that $(1-\alpha)\L_{\SR} [\psi] + \alpha\L_{\D} [\psi] \in L^2(\RR_+\times\Omega)$ and 
$$\Dlim[\psi](t,x)\cdot n(x) = 0 \quad \mbox{ for all } x\in\dO.$$

\subsection{Existence and uniqueness of weak solutions}

Throughout this section, let us write $\L_{\M} = (1-\alpha)\L_{\SR}  + \alpha\L_{\D} $ for some fixed $\alpha \in (0,1)$. 
We begin by the following well-posedness result for the stationary counterpart of \eqref{eq:LimpbMaxwell}
\begin{thm} \label{thm:stationary}
	For all $g\in L^2(\Omega)$ there exists a unique weak solution $\psi \in \mathcal{D}(\L_{\M})$ to
	\begin{equation} \label{eq:Maxwellstationary}
		\left\{ \begin{aligned}
			&\psi -\L_{\M}[\psi] = g \quad &\mbox{ in }\Omega\\
			& \Dlim [\psi]\cdot n = 0 &\mbox{ on } \dO
		\end{aligned} \right. 
	\end{equation}
	where $\mathcal{D}(\L_{\M})$ is defined in \eqref{def:DLM}.
\end{thm}
\begin{proof}
	Let us first recall some notations and results from \cite{Cesbron18} and \cite{CesbronMelletPuel}. We define the kernel $K_\Omega$, introduced in \cite[Proposition 5.1]{Cesbron18}, as
	\begin{align*}
		K_\Omega(x,y) := \frac{\gamma_1 }{|x-y|^{d+2s}} + \frac{\gamma_1 }{|(x'-y',x_d+y_d)|^{d+2s}}
	\end{align*}
	with notations $x=(x',x_d)\in \RR^{d-1}\times\RR_+$. With this kernel we can write an integration by parts formula for $\L{\SR}$: for all $\varphi ,\psi \in \mathcal{D}(\L_{\SR})$ 
	\begin{align*}
		\int_\Omega \varphi(x) \L_{\SR} [\psi](x) \d x &= \int_\Omega \psi(x) \L_{\SR} [\varphi](x) \d x \\
		&=- \frac{\gamma_1}{2}\pv \iint_{\Omega\times\Omega} \big(\varphi(y)-\varphi(x)\big)\big(\psi(y)-\psi(x)\big) K_\Omega(x,y) \d x \d y .
	\end{align*} 
	Moreover, recall \cite[Lemma 2.4]{CesbronMelletPuel} which states that $\L_{\D} [\psi] = \na_x\cdot \Dlim [\psi]$.\\
	A classical solution $\varphi$ to \eqref{eq:Maxwellstationary} then satisfy 
	\begin{equation}  \label{eq:weakstationary}
		\begin{aligned}
			&\int_{\Omega} \varphi \psi \d x + \frac{1}{2}(1-\alpha)\gamma_1\iint_{\Omega\times\Omega} \big(\varphi(y)-\varphi(x)\big)\big(\psi(y)-\psi(x)\big) K_\Omega(x,y) \d x \d y  \\
			& + \alpha \int_\Omega \Dlim [ \varphi] \cdot \na \psi \d x = \int_\Omega \varphi g \d x .
		\end{aligned}
	\end{equation}
	for all $\psi \in \mathcal{D}(\bO)$. We thus introduce the following bilinear form
	\begin{align*}
		a(\varphi,\psi) &= \int_\Omega \varphi(x) \psi(x) \d x + \frac{1}{2} (1-\alpha) \gamma_1 \iint_{\Omega\times\Omega} \big(\varphi(y)-\varphi(x)\big)\big(\psi(y)-\psi(x)\big) K_\Omega(x,y) \d x \d y  \\
		&\quad + \alpha \int_\Omega \Dlim [ \varphi] \cdot \na \psi \d x .
	\end{align*}
	We can actually decompose the operator $a$ as
	\begin{align*}
		a(\varphi,\psi) = (1-\alpha) a_{\SR}(\varphi,\psi) + \alpha a_{\D} (\varphi,\psi)
	\end{align*}
	with 
	\begin{equation*} 
		\left\{ \begin{aligned}
			a_{\SR}(\varphi,\psi) &:= \int_\Omega \varphi(x) \psi(x) \d x \\
			&\quad + \frac{\gamma_1}{2} \pv \iint_{\Omega\times\Omega} \big(\varphi(y)-\varphi(x)\big)\big(\psi(y)-\psi(x)\big) K_\Omega(x,y) \d x \d y, \\
			a_{\D} (\varphi,\psi) &:= \int_\Omega \varphi(x) \psi(x) \d x + \int_\Omega \Dlim [ \varphi] \cdot \na \psi \d x 
		\end{aligned} \right.
	\end{equation*}
	These two bilinear operators have been introduced and studied in \cite[Theorem 1.6]{Cesbron18} and \cite[Proposition 4.1]{CesbronMelletPuel} respectively. In particular they are both symmetric, continuous on $H^s(\Omega)\times H^s(\Omega)$ and are bounded above and below by the $H^s$-norm hence we have
	\begin{align*}
		c \lVert \varphi \lVert_{H^s(\Omega)} \leq a(\varphi, \varphi) \leq C \lVert \varphi \lVert_{H^s(\Omega)} \quad \forall \varphi \in H^s(\Omega)
	\end{align*}
	for some positive constants $c$ and $C$ depending only on $\Omega$ and $s$. 
	The Lax-Milgram theorem then gives existence and uniqueness of a weak solution to \eqref{eq:Maxwellstationary} in the sense that for any $g\in L^2(\Omega)$ there exists a unique $\varphi \in H^s(\Omega)$ such that 
	\begin{align*}
		a(\varphi, \psi) = \int_{\Omega} g \psi \d x \quad \forall \psi\in H^s(\Omega).
	\end{align*}
	Moreover, this weak solution satisfies in particular \eqref{eq:weakstationary} for all test function $\psi \in \mathcal{D}(\Omega)$. This means that the equation 
	\begin{align*}
		\varphi - \L_{\M} [\varphi] =  g
	\end{align*}
	holds in $\mathcal{D}'(\Omega)$, and since $\varphi$ and $g$ are in $L^2(\Omega)$ we deduce that $\L_{\M} [\varphi]\in L^2(\Omega)$. In particular, the trace $\Dlim [\varphi]\cdot n$ on $\dO$ is well defined in $H^{-1/2}(\dO)$. Finally, using \eqref{eq:weakstationary} with $\psi \in \mathcal{D}(\bO)$ we see that 
	\begin{align*}
		\Dlim [\varphi] \cdot n = 0 \quad \mbox{ on } \dO
	\end{align*}
	hence $\varphi \in \mathcal{D}(\L_{\M})$. 
\end{proof}
Theorem \eqref{thm:MaxwellSWP} then follows immediately from the Hille-Yoshida theorem.

\appendix

\section{Convergence of test functions}  \label{subsection:CVPhieps}
This appendix is devoted to the proof of convergence of test functions, used in the proofs of Theorems \ref{thm:SR}, \ref{thm:Diff} and \ref{thm:Maxwell} above. 
\begin{lemma} \label{lem:CVphieps}
	Let us consider $\psi\in L^\infty(0,+\infty;H^1(\Omega))$ and the operator $A_\eps = \eps v\cdot \na_x - \nu_0 \text{Id}$ defined on a domain 
	\[ \mathcal{D}(A_\eps) = \lbrace \phi\in L^2_F(\Omega \times\RR^d):\, v\cdot \na_x \phi\in L^2_F(\Omega\times\RR^d) \text{ and } \gamma_+ \phi = \mathcal{B}^*_\alpha[\gamma_-\phi] \text{ on } \Gamma_+ \rbrace  ,\]
	where $\B_\alpha$ is the Maxwell boundary operator with $\alpha\in[0,1]$. The function $\phi_\eps := A_\eps^{-1} [-\nu_0 \psi]$ satisfies
	\begin{align*}
		\phi_\eps \underset{\eps \to 0}{\longrightarrow} \psi \quad \mbox{ in } L^2_F((0,+\infty)\times\Omega\times\RR^d)\mbox{-strong} .
	\end{align*}
\end{lemma}
\begin{proof}
	As usual, we split the velocity integral in two parts: $|\eps v| >1$ and $|\eps v|<1$.  We begin with the former and using \eqref{def:phiepsDiffConv} we decompose $\phi_\eps - \psi$ as follows (omitting the $t$ variable for clarity): 
	\begin{align*}
		\phi_\eps (x,v) - \psi(x) &= \int_0^{\tfe} \nu_0 e^{-\nu_0 \tau} \Big( \psi(x+\eps \tau v) - \psi(x) \Big) \d \tau \\
		&+ (1-\alpha) e^{-\nu_0 \tfe} \left( A_\eps^{-1}[-\nu_0\psi](x_f,\Rxf v) - \psi(x)\Big) \right) \\
		&+ \alpha e^{-\nu_0 \tfe} c_0 \int_{w\cdot n<0} \left( A_\eps^{-1}[-\nu_0\psi](x_f,w) - \psi(x)\right) F(w) |w\cdot n| \d \tau \d w .
	\end{align*}
	Since $\tfe(x,v)>0$ and we have easily 
	\begin{align*}
		&\int_\Omega\int_{|\eps v|>1} \left(\int_0^{\tfe} \nu_0 e^{-\nu_0 \tau} \Big( \psi(x+\eps \tau v) - \psi(x) \Big) \d \tau \right)^2 F(v) n  t\d v \d x \\
		&\quad \lesssim \int_\Omega \int_{|z|>1} \int_0^{+\infty} \nu_0e^{-\nu_0 \tau} \Big( \psi(x+\tau z) - \psi(x) \Big)^2 \frac{\eps^{2s}}{|z|^{d+2s}} \d \tau\d z \d x \\
		&\quad \lesssim \eps^{2s} \| \psi\|_{L^2(\Omega)} .
	\end{align*}
	Moreover, since $\tfe(x,v)=\tau_f(x,\eps v)$, $|\Rxf v|=|v|$ and for any $(y,w)\in\Gamma_-$: 
	\begin{align*}
		A_\eps^{-1}[-\nu_0\psi] (y,w) &= \int_0^{+\infty} \nu_0e^{-\nu_0\tau} \psi(y+\eps \tau w) \d \tau \\
		&= A_1^{-1} [-\nu_0 \psi] (y,\eps w). 
	\end{align*}
	we have with the same substitution as above
	\begin{align*}
		\int_\Omega\int_{|\eps v|>1} \big(A_\eps^{-1}[-\nu_0\psi](x_f,\Rxf v) - \psi(x)\big)^2 F(v)\d v \d x \lesssim \eps^{2s}\| \psi\|_{L^2(\Omega)} . 
	\end{align*}
	Finally, with the substitution $\tau'=\eps \tau$ we have 
	\begin{align*}
		& c_0 \int_\Omega \int_{|\eps v|>1} \left( \int_{w\cdot n<0} \left( A_\eps^{-1}[-\nu_0\psi](x_f,w) - \psi(x)\right) F(w) |w\cdot n| \d w \right)^2 F(v) \d v \d x \\
		& \lesssim \int_\Omega \int_{|\eps v|>1}\int_{w\cdot n<0} \int_0^\infty \nu_0 e^{-\nu_0 \tau} \left( \psi(x_f + \eps\tau w ) - \psi(x)\right)^2 F(w) |w\cdot n| \d \tau \d w F(v) \d v \d x \\
		& \lesssim \int_\Omega \int_{|z|>1} \int_{w\cdot n<0} \int_0^\infty \nu_0 e^{-\nu_0 \tau'/\eps} \left( \psi(x_f + \tau' w ) - \psi(x)\right)^2 F(w) |w\cdot n| \frac{\d \tau}{\eps} \d w \frac{\eps^{2s}}{|z|^{d+2s}} \d z\d x \\
		&\lesssim \eps^{2s-1} \|\psi\|_{L^2(\Omega)} 
	\end{align*}
	and the convergence follows. \\
	For the integral over $|\eps v|<1$ we use a more localised analysis. We know that $\phi_\eps-\psi$ can be expressed as (omitting the $t$ variable again for clarity)
	\begin{equation} \label{eq:phieps-psi}
		\begin{aligned}
			\phi_\eps (x,v)-\psi(x) &= \int_0^{\tfe} \nu_0 e^{-\nu_0\tau} \Big( \psi(x+\eps \tau v)-\psi(x) \Big) \d \tau +  e^{-\nu_0\tfe} \Big( \psi(x_f) -\psi(x) \Big) \\
			&\quad  + e^{-\nu_0 \tfe}\Big( \mathcal{B}^*_\alpha \big( A_\eps^{-1} [-\nu_0 \psi ]\big) (x_f,v) - \psi(x_f) \Big) .
		\end{aligned}
	\end{equation}
	For the first term the computations are rather straightforward
	\begin{align*}
		&\int_{\Omega} \int_{|\eps v|<1} \bigg( \int_0^{\tfe} \nu_0 e^{-\nu_0\tau} \Big( \psi(x+\eps \tau v)-\psi(x) \Big) \d \tau\bigg)^2 F(v) \d v \d x  \\
		&\leq \int_{\Omega} \int_{|\eps v|<1}  \int_0^{\tfe} \nu_0^2 e^{-2 \nu_0 \tau} \bigg(\int_0^1 \eps \tau v \cdot \na_x \psi(x+\lambda\eps \tau v) \d \lambda \bigg)^2  F(v)\d \tau\ \d v \d x \\
		&\leq C \eps^2 \int_{\Omega} |\na_x \psi |^2 \d x  \int_0^{+\infty} \nu_0^2 \tau^2 e^{-\nu_0 \tau} \d \tau \int_{|\eps v|<1} |v|^2 F(v) \d v \\
		&\leq C\eps^{2s}  \lVert \na \psi \lVert_{L^2(\Omega)}.
	\end{align*}
	using \eqref{eq:F0} which yields
	\begin{align*}
		\int_{|\eps v|<1} |v|^2 F(v) \d v &\leq \int_{|v|<1} |v|^2 F(v) \d v + \int_{1<|v|<1/\eps} \frac{\gamma}{|v|^{d+2s-2}} \d v \leq C + C \eps^{2s-2} .
	\end{align*}
	For the second term, we can do the same computation as above and use the fact that $(\tfe)^2e^{-\nu_0 \tfe} \in L^\infty(\Omega\times\RR^d)$ to get
	\begin{align*}
		&\int_{\Omega} \int_{|\eps v|<1}  \Big( e^{-\nu_0\tfe} \Big( \psi(x_f) -\psi(x) \Big)\Big)^2 F(v) \d v \d x \\
		&\leq \int_{\Omega} \int_{|\eps v|<1}   e^{-2 \nu_0 \tfe} \bigg(\int_0^1 \eps \tfe v \cdot \na_x \psi(x+\lambda\eps \tfe v) \d \lambda \bigg)^2  F(v)\d \tau \d v \d x \\
		&\leq C \eps^2 \int_{\Omega} |\na_x \psi |^2 \d x\int_{|\eps v|<1} |v|^2 F(v) \d v \\
		&\leq C\eps^{2s}  \| \na \psi \|_{L^2(\Omega)}.
	\end{align*}
	Next, we split the boundary operator into its specular and its diffusive part. For the specular part we prove convergence similarly to the previous two terms: 
	\begin{align*}
		&\int_{\Omega} \int_{|\eps v|<1} \bigg( e^{-\nu_0\tfe} \int_0^{+\infty} \nu_0 e^{-\nu_0 \tau} \Big(\psi(x_f+\eps \tau \Rxf v) - \psi(x_f) \Big) \d \tau \bigg)^2 F(v) \d v \d x  \\
		& \lesssim \int_{\Omega} \int_{|\eps v|<1} e^{-2\nu_0\tfe} \int_0^{+\infty} \nu_0^2e^{-2\nu_0\tau} \bigg( \int_0^1 \eps \tau \Rxf v \cdot \na \psi(x_f + \lambda \eps \tau \Rxf v) \d \lambda \bigg)^2 F(v) \d v \d x\\
		&\lesssim \eps^{2s} \| \na \psi \|_{L^2(\Omega)} 
	\end{align*}
	using the fact that $|\Rxf v| = |v|$. For the diffusive part, the integral over $|\eps v|<1$ will not play a important role, we focus instead on the integral in $w$ to write
	\begin{align*}
		&\int_\Omega \left( c_0 \int_{w\cdot n<0} \big( A_\eps^{-1}[-\nu_0 \psi] (x_f, w)-\psi(x_f) \big) |w\cdot n|F(w) \d w \right)^2 \d x \\
		&\quad \lesssim \int_\Omega \bigg( \int_{w\cdot n<0, |\eps w|<1} \big( A_\eps^{-1}[-\nu_0 \psi] (x_f, w)-\psi(x_f) \big)^2 |w\cdot n|F(w) \d w \\
		&\qquad + \int_{w\cdot n<0, |\eps w|>1} \big( A_\eps^{-1}[-\nu_0 \psi] (x_f, w)-\psi(x_f) \big)^2 |w\cdot n|F(w) \d w \bigg) \d x \\
	\end{align*}        
	Using the same techniques and controls as above, we get:
	\begin{align*}    
		&\int_\Omega \left( c_0 \int_{w\cdot n<0} \big( A_\eps^{-1}[-\nu_0 \psi] (x_f, w)-\psi(x_f) \big) |w\cdot n|F(w) \d w \right)^2 \d x \\
		&\quad \lesssim \int_\Omega \int_{z\cdot n<0, |z|>1} \big( A_1^{-1}[-\nu_0\psi](x_f,z) -\psi(x_f)\big)^2 \frac{|z\cdot n|}{\eps} \frac{\eps^{2s}}{|z|^{d+2s}} \d z \d x\\
		&\qquad + \int_\Omega \int_{|\eps w|<1} \int_0^{+\infty} \nu_0 e^{-\nu_0 \tau} \left(\int_0^1 \eps \tau w\cdot\na \psi(x_f+\lambda \eps \tau w) \d \lambda \right)^2 |w\cdot n| F(w) \d \tau \d w \d x \\
		&\quad \lesssim \eps^{2s-1} \| \psi\|_{L^2(\Omega)} + \eps^2 \| \na\psi \|_{L^2(\Omega)} \int_{|\eps w|<1} |w|^3 F(w) \d w \\
		&\quad \lesssim \eps^{2s-1} \| \psi\|_{L^2(\Omega)} + \eps^{2s-1} \| \na\psi \|_{L^2(\Omega)} .
	\end{align*}
	This concludes the proof of convergence. 
\end{proof}
As an immediate corollary of Lemma \ref{lem:CVphieps}, we have
\begin{cor}
	For all $\psi \in W^{1,\infty}(0,+\infty; H^1(\Omega))$ and $\phi_\eps$ defined in Lemma \ref{lem:CVphieps} we have
	\begin{align*}
		\pa_t \phi_\eps \underset{\eps \to 0}{\longrightarrow} \pa_t \psi \quad \mbox{ in } L^2_F((0,+\infty)\times\Omega\times\RR^d)\mbox{-strong} .
	\end{align*}
\end{cor}


\bibliographystyle{siam}
\bibliography{biblio.bib}

\end{document}